\documentclass{amsart}
\usepackage{amsthm,amsfonts,amssymb,amsmath,amscd
,color
}
\usepackage{textcomp}
\usepackage{appendix}
\usepackage[centerlast]{subfigure}
\usepackage[all,2cell, cmtip]{xy} \UseAllTwocells
\DeclareMathOperator{\Aut}{Aut}

\DeclareMathOperator{\id}{id}

\DeclareMathOperator{\Coker}{Coker}
\DeclareMathOperator{\Hom}{Hom}
\DeclareMathOperator{\Mor}{Mor}
\DeclareMathOperator{\Map}{Map}

\DeclareMathOperator{\Ker}{Ker}
\DeclareMathOperator{\Ob}{Ob}

\DeclareMathOperator{\Vect}{\mathbf{Vect}}
\DeclareMathOperator{\ev}{ev}
\DeclareMathOperator{\2ch}{Kom(\Pic)}

\newcommand{\Ch}{{\v C}ech}
\DeclareMathOperator{\cu}{C(X;\mathfrak{U}_I)}
\newcommand{\ChZG}{hermitian line $0$-gerbe}
\newcommand{\ChG}{flat hermitian line $1$-gerbe}
\newcommand{\ChC}{hermitian line $1$-cocycle}
\newcommand{\ChGG}{flat hermitian line $2$-gerbe}

\DeclareMathOperator{\Sing}{Sing}

\newtheorem{thm}{Theorem}[section]

\newtheorem{prop}[thm]{Proposition}

\theoremstyle{definition}
\newtheorem{df}[thm]{Definition}
\newtheorem{ex}[thm]{Example}

\theoremstyle{remark}

\newtheorem*{rem}{Remark}
\newtheorem{ack}{Acknowledgments}

\def\Z{\mathbb Z}

\def\Pic{\mathcal P\textit{ic}}
\def\A{\mathcal A}
\def\C{\mathcal C}
\def\D{\mathcal D}
\def\G{\mathcal G}
\def\B{\mathcal B}
\def\L{\mathcal L}
\def\O{\mathcal O}
\newcommand{\del}{\partial}

\begin{document}

\title[Categorification of Dijkgraaf-Witten Theory]{Categorification
  of Dijkgraaf-Witten Theory}
\author[A. Sharma]{Amit Sharma}
\email{sharm121@umn.edu}
\address {School of Mathematics\\University of Minnesota\\
  Minneapolis, MN 55455, USA}
\author[A. A. Voronov]{Alexander A. Voronov}
\email{voronov@umn.edu}
\address {School of Mathematics\\University of Minnesota\\
  Minneapolis, MN 55455, USA, and Kavli IPMU (WPI), UTIAS, The
  University of Tokyo, Kashiwa, Chiba 277-8583, Japan}

\thanks{This work was supported by the World Premier International
  Research Center Initiative (WPI), MEXT, Japan, and a
  grant from the Simons Foundation (\#282349 to A.~V.).}

\date{February 8, 2016}

\begin{abstract}
  The goal of the paper is to categorify Dijkgraaf-Witten (DW) theory,
  aiming at providing foundation for a direct construction of DW
  theory as an Extended Topological Quantum Field Theory. The main
  tool is cohomology with coefficients in a Picard groupoid, namely
  the Picard groupoid of hermitian lines.
\end{abstract}

\maketitle

\tableofcontents

\section{Introduction}

R.~Dijkgraaf and E.~Witten in \cite{dw} constructed a gauge theory
with a finite gauge group $G$ as a ``toy model,'' a tool for studying
more general gauge theories with compact gauge groups. Their goal was
to describe this theory, known as \emph{DW theory}, as a Topological
Quantum Field Theory (TQFT), \emph{i.e}., a functor on the category of
3-dimensional (3d) cobordisms to that of vector spaces, starting with
an action given by a cocycle $\alpha \in Z^3(G; U(1))$. Dijkgraaf and
Witten indicated that the vector space $\Phi(Y)$ corresponding to a
closed oriented 2d manifold $Y$ was closely related to the set
$\Hom(\pi_1 (Y), G)/G$ of equivalence classes of principal $G$-bundles
over $Y$ and that it could be constructed by cutting the surface $Y$
into pairs of pants, as $\Phi$ was expected to be a functor.  The
linear map $\Phi(X): \del_- X \to \del_+ X$ corresponding to a 3d
oriented cobordism $X$ between closed manifolds $\del_- X$ and $\del_+
X$ depended on such choices as the choice of a map $\Hom (\pi_1 (X,
x_0), G) \to \Map(X, BG)$, the choice of a basepoint $x_0$, the choice
of a chain, via triangulation, representing the relative fundamental
cycle $[X] \in H_3(X, \del X; \Z)$, which was interpreted as ``lattice
gauge theory.'' One can say that, from the categorical point of view,
Dijkgraaf and Witten constructed a TQFT functor on a certain
subcategory of cobordisms decorated with appropriate extra structure
utilized in their constructions. They used an orbifold approach to
taking the homotopy quotient by $G$, that is to say, worked with the
$G$-set $\Hom(\pi_1(Y), G)$.

D.~Freed and F.~Quinn in \cite{fq, freed} streamlined the construction
of the TQFT functor $\Phi$, so that $\Phi(X)$ would no longer depend
on the choice of a representative of the fundamental cycle $[X]$ and
thereby would produce a TQFT functor on the category of
cobordisms. They also generalized the construction to $n$-dimensional
cobordisms. Their main tool was to define pairings between cocycles in
$Z^{n+1} (Y, U(1))$ and cycles $Z_n (Y, \Z)$ and between $Z^{n+1} (X,
U(1))$ and cycles $Z_{n+1} (X, \del X; \Z)$, resembling but certainly
different from cap product, which would not even be defined because of
dimension considerations. Freed and Quinn introduced the idea of an
\emph{invariant section} of a flat hermitian line bundle over a
groupoid. This is a particular case of the idea of the limit of a
functor, and in this context, is akin to taking a global section.

J.~Lurie in \cite{heuts-lurie} sketched a different construction of
Dijkgraaf-Witten theory. Rather than using the orbifold $\Hom
(\pi_1(Y), G)/G$, he modeled the set of equivalence classes of
principal $G$-bundles on the mapping space $\Map(Y,BG)$. Given a
cohomology class $\alpha \in H^{n+1}(BG; U(1))$ and a closed oriented
$n$-manifold $Y$, he used a ``push-pull'' construction $\pi_* \ev^*
\alpha \in H^1 (\Map (Y, BG); U(1))$ for the diagram
\[
\begin{CD}
Y \times \Map (Y,BG) @>\ev>> BG\\
@V{\pi}VV\\
\Map (Y, BG)
\end{CD}
\]
to obtain a hermitian line bundle $\L_Y$ over $Y$. Then he defined the
TQFT functor $\Phi$ on objects by taking the space
\[
\Phi(Y) := H^0(\Map (Y, BG), \L_Y)
\]
of global sections. He used \emph{ambidexterity}, a natural
isomorphism
\[
H^0 (\Map (Y, BG), \L_Y) \xrightarrow{\sim} H_0 (\Map (Y, BG), \L_Y),
\]
to produce a linear map
\[
\Phi(X): \Phi (\del_- X) \to \Phi (\del_+ X),
\]
using push-pull again, now along the diagram
\[
\Map (\del_- X, BG) \xleftarrow{p_-} \Map (X, BG) \xrightarrow{p_+}
\Map (\del_+ X, BG).
\]
Lurie's construction deliberately avoided the following subtlety. The
hermitian line bundle $\L_Y$ is determined by the cohomology class
$\alpha$ only up to isomorphism. Starting with a cocycle $\alpha \in
Z^{n+1}(BG; U(1))$ would partially fix the problem, because the
resulting cocycle $\pi_* \ev^* \alpha \in Z^1(\Map (Y, BG); U(1))$ is
not quite the same as a hermitian line bundle: isomorphic, but
different hermitian line bundles may correspond to the same cocycle,
whereas the cocycle is determined by a hermitian line bundle only up
to condoundary. Moreover, the push-pull cocycle $\pi_* \ev^* \alpha$
will depend on the choice of a cycle representing the fundamental
class $[Y] \in H_n (Y; \Z)$.

In the current paper, we replace the coefficient group $U(1)$ with an
equivalent Picard groupoid, namely the Picard groupoid $\L$ of
hermitian lines, and notice that an object of $H^0(M, \L)$ is exactly
a flat hermitian line bundle over $M$, see Section~\ref{gerbes}.

The paper \cite{fhlt} attempted the construction of an Extended
Topological Quantum Field Theory (ETQFT), which is defined on
cobordisms with corners, rather than boundary, and a generalization of
the DW theory to the case of a compact group $G$. The construction
utilizes the Cobordism Hypothesis, which asserts that an ETQFT is
determined by its value on zero-dimensional manifolds. The
two-dimensional case of the cobordism hypothesis was proved by
C.~J. Schommer-Pries in \cite{SP}, and the full version was proven by
Lurie in \cite{JL1}. However, Freed, Hopkins, Lurie, and Teleman
emphasize the importance of a direct construction, which has not been
been done yet.

This paper arose from the authors' trying to find an approach to this
hypothetical direct construction of an ETQFT. In the process we have
realized that Freed and Quinn's pairing makes sense as a cohomological
operation, cap product, if the group $H^{n+1} (Y; U(1))$ is replaced
with cohomology $H^n (Y; \L)$ with coefficients in the Picard groupoid
$\L$ of hermitian line bundles. Categorifying the coefficients goes
along with lowering the cohomological degree, thus opening a way to
defining cap products as well as extending the TQFT to an ETQFT by
further categorification to higher Picard groupoids and higher gerbes.

Another novel feature of our approach is that we do not use
ambidexterity, but rather a transfer map in the context of cohomology
with coefficients in Picard groupoids. In principle, one can view the
transfer map as an avatar of ambidexterity, but it might be argued
that using an avatar is less demanding than engaging the full power of a
deity.

\begin{ack}
  We thank Jim Stasheff for valuable comments on the first version of
  the manuscript.  A.~V. gratefully acknowledges support from the
  Simons Center for Geometry and Physics, Stony Brook University, and
  the Graduate School of Mathematical Sciences, The University of
  Tokyo, at which some of the research for this paper was performed.
\end{ack}

\section{Setup}

We will consider (flat) hermitian line gerbes over simplicial sets. To
deal with gerbes over manifolds and topological spaces, we will
associate simplicial sets to them in a standard way: by taking
singular simplices or the nerve of an open cover. Flat hermitian line
gerbes are analogous to more traditional gerbes over topological
spaces with the constant sheaf $U(1)$ as the band, whether given as
stacks of groupoids, via gluing (descent) data, or as higher bundles,
\cite{breen-messing,brylinski,Moerdijk,murray}. We will take the
liberty of omitting the adjective ``flat'' when referring to flat
hermitian line bundles and gerbes.

We will describe cohomology with coefficients in Picard groupoids over
simplicial sets and later apply this construction to cobordisms, which
are manifolds, rather than simplicial sets. This may be done by
working with the simplicial set of singular chains associated to the
cobordism or by using the nerve of a sufficiently fine open covering,
see examples in Section~\ref{gerbes}.

\subsection{Cohomology with coefficients in Picard groupoids}
\label{coh}

A \emph{Picard groupoid} is a symmetric monoidal groupoid in which
every object is invertible, up to isomorphism, with respect to the
tensor product, which, by a slight abuse of notation, we denote
$+$. More precisely, for each object $s$ of a Picard groupoid $\A$,
the functors $t \mapsto s + t$, and $t \mapsto t + s$ define
autoequivalences of $\A$ as a category. In this case, one can define a
functor $\A \to \A$, $s \mapsto -s$, and natural isomorphisms
\[
m= m_s: s + (-s) \to 0, \qquad n=n_s: (-s) + s \to 0
\]
such that $l_s (m_s + \id_s) = r_s (\id_s + n_s) \alpha_{s,-s,s}$ for
all objects $s$ of $\A$, where $0$ is the zero (also known as unit)
object of $\A$ and
\begin{gather}
\label{structure1}
  \alpha_{s,t,u}: (s+t)+u \to s + (t+u) \qquad \text{and}\\
\label{structure2}
  l_s: 0+s \to s, \quad r_s: s + 0 \to s
\end{gather}
are the natural transformations of the monoidal structure on $\A$. We
will assume that $-0 = 0$, $m_0 = r_0$, and $n_0 = l_0$. Another
structure natural transformation is a symmetry:
\[
\beta_{s,t}: s+t \to  t+s,
\]
making $\A$ to be a symmetric monoidal category.  Given a Picard
groupoid $\A$, let $\pi_0 (\A)$ denote the abelian group of its
connected components and $\pi_1 (\A)$ denote the abelian group of
automorphisms of the zero object.

A \emph{homomorphism} between two Picard groupoids $\A$ and $\B$ is a
functor $F: \A \to \B$ and an assignment of a \emph{coherence
  morphism} which is an arrow of $\B$, $\phi^F_{s,t}:F(s) + F(t) \to
F(s+t)$, to every pair of objects $s,t \in \A$ which is natural in
both variables $s$ and $t$ such that the assignment respects the
symmetry natural transformations $\beta$ of $\A$ and $\B$ in the
following sense:
\[
F(\beta_{s,t}) \circ \phi^F_{s,t} =
\phi^F_{t,s} \circ \beta_{F(s),F(t)} 
\]
and also respects the associativity in the following sense:
\[
 \phi^F_{s,t+u} \circ (id_{F(s)} + \phi^F_{t,u})\circ \alpha^{\B}_{F(s),F(t),F(u)}
 = F(\alpha^{\A}_{s,t,u}) \circ \phi^F_{s+t,u} \circ (\phi^F_{s,t} + id_u),
\]
for each triple of objects $s,t,u \in \A$ and where $\alpha^\A$
and $\alpha^\B$ are the associativity natural transformations of
$\A$ and $\B$ respectively.

A homomorphism between two Picard groupoids $F: \A \to \B$ will be
called a \emph{strict homomorphism} if the coherence morphisms
$\phi^F_{s,t}$ are identities for all pairs $s,t \in \A$ and $F(0) =
0$.

Given
two homomorphisms $F$ and $F': \A \to \B$, a \emph{monoidal natural
transformation from $F$ to $F'$} is a natural transformation
$\theta: F \Rightarrow F'$ which is compatible with the coherence
morphisms of both homomorphisms $F$ and $F'$ in the following
sense:
\[
\phi^{F'}_{s,t} \circ (\theta_s + \theta_t) = \theta_{s+t} \circ \phi^F_{s,t},
\]
for all pairs $s, t \in \A$.


%
%

Given any two Picard groupoids $\A$ and $\B$, the category whose
objects are all homomorphisms from $\A$ to $\B$ and whose morphisms
are monoidal natural transformations
between these homomorphisms has the structure of a Picard groupoid
which we denote by $[\A,\B]$, see \cite{schmitt} for a detailed proof of this
assertion. One can associate another Picard groupoid with $\A$ and $\B$
which we denote by $\A \otimes \B$, and which will be called the
\emph{tensor product}. We will not recall its construction, which
is rather elaborate, see \cite{schmitt}, but mention that
the tensor product 2-functor is determined by an adjunction
\[
[\A, [\B, \C]] \xrightarrow{\sim} [ \A \otimes \B, \C]
\]
in the bicategory of Picard groupoids. This bicategory also has a unit
object $I$ for the monoidal structure. The bicategory of Picard
groupoids, not only has an internal hom as indicated above, but it has
the structure of a $\Pic$-category, see appendix
\ref{pic-category-appendix} for a definition of a
$\Pic$-category. More precisely, Picard groupoids, homomorphisms
between Picard groupoids and monoidal natural transformations between
homomorphisms form a $\Pic$-category which we denote by
$\Pic$. Further, $\Pic$ is the archtype example of a
$\Pic$-category. Our point of view on $\Pic$ is that it is the analog
of the category of Abelian groups, $\mathbf{Ab}$, in the world of
bicategories.

The groupoid of lines, i.e., one-dimensional vector spaces, and
$G$-torsors for a given abelian group $G$ have natural structures of
Picard groupoids with respect to tensor products and the product of
torsors over $G$, respectively. We will later focus our attention on
the Picard groupoid $\L$ of hermitian lines, where the hermitian form
on the tensor product of hermitian lines is the tensor product of the
hermitian forms one each line.

Let $X_\bullet$ be a simplicial set and $\A$ be a Picard groupoid. We
will define cohomology $H^\bullet (X_\bullet, \A)$ of $X_\bullet$ with
values in $\A$, following \cite{carrasco-martinez-moreno} and
\cite{dr-mm-v}. Similar cohomology may be defined for topological
spaces and, more generally, with coefficients in sheaves of Picard
groupoids.

Let us associate with $X_\bullet$ and $\A$ a \emph{cosimplicial Picard
  groupoid}, that is to say, a cosimplicial object in the category of
Picard groupoids, defined as the ``mapping space'' $\A^{X_\bullet} :=
\Map (X_\bullet, \A)$: for each $n \ge 0$, we define the Picard
groupoid $\A^{X_n}$ whose objects are maps $X_n \to \Ob \A$, morphisms
are maps $X_n \to \Mor \A$, and the tensor product and morphism
composition are defined ``point-wise.''  The cosimplicial structure is
comprised of homomorphisms of Picard groupoids:
\[
\xymatrix{ \A^{X_0} \ar@/_/[rr]|{d_0^*} \ar@/_/@<-1ex>[rr]_{d_1^*} &&
  \A^{X_1} \ar@/_/[ll]|{s_0^*} \ar@/_/[rr]|{d_0^*} \ar@/_/@<-1ex>[rr]
  \ar@/_/@<-2ex>[rr]_{d_2^*} && \A^{X_2} \ar@/_/[ll]|{s_0^*}
  \ar@/_/@<-1ex>[ll]_{s_1^*} \ar@/_/[rr]|{d_0^*} \ar@/_/@<-1ex>[rr]
  \ar@/_/@<-2ex>[rr] \ar@/_/@<-3ex>[rr]_{d_3^*} && \ar@/_/[ll]|{s_0^*}
  \ar@/_/@<-1ex>[ll] \ar@/_/@<-2ex>[ll]_{s_2^*} \dots },
\]
where the coface and codegeneracy homomorphisms $d_i^*: \A^{X_n} \to
\A^{X_{n+1}}$ and $s_j^* : \A^{X_{n+1}} \to \A^{X_n}$ are obtained by
composition with the face maps $d_i: X_{n+1} \to X_n$ and degeneracy
maps $s_j: X_n \to X_{n+1}$ of the simplicial set $X_\bullet$,
respectively.

Now, by taking alternating sums, we obtain a \emph{$($cochain$)$
  complex of Picard groupoids}:
\begin{equation*}
  C^\bullet (X_\bullet, \A): \xymatrix{ 0 \ar[r] &
    \A^{X_0} \ar[r]^d \rruppertwocell<10>^0{\omit}
    & \A^{X_1} \ar[r]^d \rrlowertwocell<-10>_0{<3>\chi} & \A^{X_2}
    \lltwocell<\omit>{<3>\chi} \ar[r]^d \rruppertwocell<10>^0{\omit} &
    \A^{X_3} \ar[r]^d & \lltwocell<\omit>{<3>\chi} \dots }
\end{equation*}
with $d = \sum_{i=0}^{n+1} (-1)^i d_i^* : \A^{X_n} \to \A^{X_{n+1}}$
and a monoidal transformation $\chi: d^2 \Rightarrow 0$, obtained in a
unique way from the structure isomorphisms $\alpha$, $m$ and $n$. This
system of coboundary homomorphisms $d$ and monoidal transformations
$\chi$ is \emph{coherent}, i.e., $\chi d = d \chi$ as 2-cells $d^3
\Rightarrow 0$. Let $\2ch$ denote the $\Pic$-category of complexes
of Picard groupoids. The objects of $\2ch$ are complexes of
Picard groupoids. A 1-morphism between $\A^{\bullet},
\B^{\bullet} \in Ob(\2ch)$, pictured below:
\begin{equation*}
  \xymatrix{ \dots  
    \A^{n-1} \ar[d]_{f^{n-1}} \ar[r]^{d_{\A}} \rruppertwocell<10>^0{\omit}
    & \A^{n} \ar[d]_{f^{n}} \ar[r]^{d_{\A}} \ar@{=>}_{\phi^{n}}(12,-5)*{};(9,-9)*{}
    & \A^{n+1} \ar[d]^{f^{n+1}} \lltwocell<\omit>{<3>\chi_{\A}}
    \ar@{=>}^{\phi^{n+1}}(29,-5)*{};(26,-9)*{} \dots
    \\
    \dots  
    \B^{n-1} \ar[r]_{d_{\B}} \rrlowertwocell<-10>_0{\omit}
    & \B^{n} \ar[r]_{d_{\B}}  & \B^{n+1} \lltwocell<\omit>{<-3>\chi_{\B}} \dots },
\end{equation*}
is a pair $F = (f, \phi)$, where $f$ is a sequence of homomorphisms
$f^{n}: \A^{n} \rightarrow \B^{n}$ and $\phi$ is a sequence of
monoidal natural transformations $\phi^{n}:f^{n}d_{\A} \Rightarrow
d_{\B}f^{n-1}$ in $\Pic$, satisfying the following coherence conditions
$\phi^{n+1}d_{\A} = d_{\B}\phi^{n}$ and $(f^{n+1}\chi_\A) \circ
(\phi^{n+1}d_\A) \circ (d_\B \phi^{n}) = \chi_\B f^{n-1}$.
A 2-morphism $(f, \phi)\Rightarrow(f', \phi')$ is a sequence
$\lbrace \gamma_n \rbrace_{n \in \mathbb{Z}}$, where $\gamma_n:f^n
\Rightarrow f'^n$ is a monoidal natural transformation, for all $n \in \mathbb{Z}$,
and the following coherence condition is satisfied:
$(\gamma_{n+1}d_\A) \circ \phi_n = \phi'_n \circ (d_\B \gamma_n)$.
It would be useful to describe an alternative, equivalent, notion of
a $2$-morphism in $\2ch$ which is a generalization of \emph{cochain homotopy}
to the Picard groupoid context. In this notion, a $2$-morphism is
also a pair $H = (h, \psi)$, where $h^{n}: \A^{n} \rightarrow
\B^{n-1}$ and $\psi$ is a sequence of monoidal natural transformations
$\psi^{n}: d_{\B}h^{n} + f^{n} \Rightarrow f'^{n} + h^{n+1}d_{\A}$
satisfying an obvious coherence condition.
We leave the establishment of an equivalence between the two notions
of $2$-morphisms in $\2ch$ as an excercise for an interested reader.

The cohomology $H^\bullet (X_\bullet, \A)$ of $X_\bullet$ with
coefficients in a Picard groupoid $\A$ is defined as the cohomology of
the complex $(\A^{X_\bullet}, d, \chi)$ of Picard groupoids. The
cohomology of a complex of Picard groupoids may be defined as
follows. In principle, to define the $n$th cohomology $H^n (X_\bullet,
\A)$, we want to take the kernel $\Ker d$ of the homomorphism
$d:\A^{X_n} \to \A^{X_{n+1}}$ and then the cokernel of the
homomorphism $d': \A^{X_{n-1}} \to \Ker d$ induced by $d:\A^{X_{n-1}}
\to \A^{X_{n}}$, but these need to be defined in a suitable
categorified sense. In particular, the kernels, cokernels, and
cohomology will depend on two subsequent coboundary homomorphisms $d$
as well as $\chi$ and be Picard groupoids. The objects of the category
$\Ker (d,\chi)$ (of $n$-\emph{cocycles}) are pairs $(a,\phi)$ in which
$a$ is an object of $\A^{X_n}$ and $\phi: da \to 0$ is a morphism in
$\A^{X_{n+1}}$ satisfying a \emph{cocycle condition}:
\[
d (\phi) = \chi_a: d^2 (a) \to 0.
\]
A morphism $(a,\phi) \to (a',\phi')$ in $\Ker (d,\chi)$ is given by a
morphism $f: a \to a'$ in $\A^{X_n}$ such that $ \phi'\circ d (f) =
\phi$. The monoidal structure on $\Ker (d,\chi)$ is inherited from
that of $\A$. The kernel $\Ker (d,\chi)$ naturally participates in a
complex of Picard groupoids, as follows:
\begin{equation*}
  \xymatrix{ \A^{X_{n-2}} \ar[r]^d \rruppertwocell<10>^0{\omit}
    & \A^{X_{n-1}} \ar[r]^{d'}  & \Ker (d,\chi) \lltwocell<\omit>{<3>\chi'}
  }.
\end{equation*}
The cohomology $H^n (X_\bullet, \A)$ is defined as the cokernel
$\Coker (d',\chi')$ in this complex. The cokernel $\Coker (d',\chi')$
is a Picard category whose objects are the same as those of $\Ker
(d,\chi)$, i.e., of the type $(a,\phi)$, where $a$ is an object of
$\A^{X_{n}}$ and $\phi: da \to 0$ is a morphism in $\A^{X_{n+1}}$
satisfying the cocycle condition above. A morphism $(a,\phi) \to
(a',\phi')$ in $\Coker (d',\chi')$ is given by an equivalence class of
pairs $(b,f)$, where $b$ is an object of $\A^{X_{n-1}}$ and $f:
(a,\phi) \to (d'b + a', \chi_b + \phi')$ is a morphism in $\Ker
(d,\chi)$. Two morphisms $(b,f)$ and $(b',f'): (a,\phi) \to
(a',\phi')$ are \emph{equivalent}, if there is a pair $(c,g)$ with $c$
being an object of $\A^{X_{n-2}}$ and $g: b \to dc + b'$ a morphism in
$\A^{X_{n-1}}$ such that the following diagram commutes:
\[
\begin{CD}
  a @>f>> d'b + a' @>{d'(g)+ \id}>> (d'dc
  +d'b')+a'\\
  @Vf'VV & & @VV{\alpha}V \\
  d'b' + a' @<<{l}< 0 + (d'b'+a') @<<{\chi'_c + \id + \id}< dd'c
  +(d'b'+a') .
\end{CD}
\]
One can check that $ \pi_0 (H^n (X_\bullet, \A)) \cong \pi_1
(H^{n+1}(X_\bullet, \A))$.

The \emph{simplicial homology of a simplicial set $X_\bullet$ with
  coefficients in a Picard groupoid $\A$} may be defined similarly by
looking at the \emph{simplicial Picard groupoid} $\A X_\bullet$ whose
$n$-simplices are formal ``linear combinations'' $a_1 s_1 + \dots +
a_k s_k$ of pairwise distinct elements $s_1, \dots, s_k$ in $X_n$ with
coefficients $a_1, \dots, a_k$ in $\A$. Perhaps, a better way of
looking at $\A X_\bullet$ is to view it as $\A$-valued functions on
$X_\bullet$ with finite support and apply the same treatment to it as
that for $\A^{X_\bullet}$. In particular, summing up the face
homomorphisms gives rise to a \emph{chain complex of Picard groupoids}
$C_\bullet (X_\bullet, \A)$, which determines the homology Picard
groupoids $H_n (X_\bullet, \A)$ for $n \ge 0$.

When $A$ is an abelian group, we will think of it as a \emph{discrete}
Picard groupoid, denoted $A[0]$, with $A$ being the set of objects and
identities being the only morphisms, so as $\pi_0 (A[0]) = A$ and
$\pi_1 (A[0]) = 0$. Then the (co)homology with coefficients in the
Picard groupoid $A[0]$ will be related to the usual simplicial
(co)homology with coefficients in the group $A$ as follows:
\begin{eqnarray*}
\pi_0 H^\bullet (X_\bullet; A[0]) & = & H^\bullet (X_\bullet; A),\\
\pi_0 H_\bullet (X_\bullet; A[0]) & = & H_\bullet (X_\bullet; A).
\end{eqnarray*}

\subsection{Relative cohomology}
\label{relativecohomology}
Let $\A \in \Pic$, let $X_{\bullet}$ be a simplicial set, let
$Y_{\bullet} \subset X_{\bullet}$ be a simplicial subset. There is an
inclusion map $Y_{\bullet} \hookrightarrow X_{\bullet}$ in that category
of simplicial sets. This inclusion induces a $1$-morphism
\[
i_{\bullet}:C_{\bullet}(Y_{\bullet}, \A) \hookrightarrow C_{\bullet}(X_{\bullet},
\A)
\]
in $\2ch$. We define relative homology
$H_{\bullet}(X_{\bullet}, Y_{\bullet}, \A)$ to be the homology of the
2-chain complex given by the cokernel of $i_{\bullet}$ in $\2ch$.
We call this 2-chain complex, given by the cokernel, a \emph{relative 2-chain
complex} so $H_{\bullet}(X_{\bullet}, Y_{\bullet}, \A)$ is the homology of
the relative 2-chain complex $C_{\bullet}(X_{\bullet}, Y_{\bullet}, \A)$.
The $nth.$ degree of the relative
2-chain complex is the Picard groupoid given by the cokernel, in the
category of Picard groupoids, of the map
$i_{n}:C_{n}(Y_{\bullet}, \A) \hookrightarrow C_{n}(X_{\bullet}, \A)$.
Relative cohomology is defined similarly, $H^{\bullet}(X_{\bullet},
Y_{\bullet}, \A)$ is the cohomology of the relative 2-cochain complex
given by the cokernel of the following map, induced by the inclusion
$Y_{\bullet} \hookrightarrow X_{\bullet}$
\[
i^{\bullet}:C^{\bullet}(Y_{\bullet}, \A) \to C^{\bullet}(X_{\bullet},\A).
\]
The objects of $i^{\bullet}(C^{n}(Y_{\bullet}, \A))$ are those functions,
$X_n \to Ob(\A)$, which  vanish outside of $Y_{n}$.
$C^{n}(X_{\bullet}, Y_{\bullet}, \A)$ is a Picard subgroupoid
of $C^{n}(X_{\bullet}, \A)$ whose objects are the same as those of
$C^{n}(X_{\bullet}, \A)$. A morphisms in $C^{n}(X_{\bullet},
Y_{\bullet}, \A)$ is a certain equivalence class of morphisms in
$C^{n}(X_{\bullet}, \A)$.  The cokernel also gives a $1$-morphism, in $\2ch$,
$p^{\bullet}:C^{\bullet}(X_{\bullet}, \A) \to C^{\bullet}(X_{\bullet}, Y_{\bullet},
\A)$ and a $2$-morphism $\phi^{\bullet}: p^{\bullet} \circ i^{\bullet} \Rightarrow
0:C^{\bullet}(Y_{\bullet}, \A) \to C^{\bullet}(X_{\bullet}, Y_{\bullet}, \A)$,
where $0$ is the zero homomorphism. If $\alpha
\in Ob(C^{n}(Y_{\bullet}, \A))$ then the natural transformation
$\phi^n$ assigns to $\alpha$, a morphism $i(\alpha) \to 0$ in
$C^{n}(X_{\bullet}, Y_{\bullet}, \A)$.  In other words those objects
of $C^{n}(Y_{\bullet}, \A)$ are isomorphic to the zero object in
$C^{n}(X_{\bullet}, Y_{\bullet}, \A)$.

\subsection{Functoriality}
\label{functoriality}

The $n$th cohomology (and $n$th homology) defined above is a functor of
$\Pic$-categories, see appendix \ref{pic-category-appendix},
$H^{n}:\2ch \rightarrow \Pic$. Moreover, every $F \in
\Mor_{\2ch}(\A^{\bullet}, \B^{\bullet})$ determines a morphism
$H^{\bullet}(F) \in \Mor_{\2ch}( H^\bullet (\A^\bullet ),\newline H^{\bullet}
(\B^\bullet))$ on cohomology. This fact follows from properties of
relative kernels and cokernels; for a direct proof of this fact see
\cite{dr-mm-v}.

Note that the described cohomology and homology are (strictly)
functorial with respect to simplicial maps. If $\A$ is Picard
groupoind and $f: X_\bullet \to Y_\bullet$ is a simplicial map, then
we get a strict morphism between the corresponding cochain complexes
of Picard groupoids $f^*: C^\bullet (Y_\bullet, \A) \to C^\bullet
(X_\bullet, \A)$, which yields a strict morphism on cohomology $f^*:
H^n (Y_\bullet, \A) \to H^n (X_\bullet, \A)$ for $n \ge 0$. Moreover,
a simplicial homotopy between two simplicial maps induces a monoidal
natural transformation on cohomology, \emph{cf}.\ \cite[Proposition
2.1]{bullejos-carrasco-cegarra} and \cite[Proposition
2.3(i)]{carrasco-martinez-moreno} and the discussion of 2-morphisms in
$\2ch$ in Section~\ref{coh}. The same statements are true for
homology.

\subsection{The long 2-exact sequence}

We begin this subsection by recalling the notion of a short 2-exact
sequence of Picard groupoids. Here we will only recall this notion in
a subcategory of $\Pic$ which has the same objects as $\Pic$ and whose
morphisms are homomorphisms which preserve the unit of addition.  For
the general case see \cite{BV,Rouss}. A complex
\[
\xymatrix{ 
0 \ar[r] &\A \ar[r]_F \rruppertwocell<10>^0{\omit}  &\C \ar[r]_G
&\B \lltwocell<\omit>{<3>\phi} \ar[r] &0
 \\
}
\]
is called a \emph{short $2$-exact sequence of Picard groupoids} if the
unique morphism $\overline{G}: \Coker(F, \id_0) \to \B$ is full and
faithful and further, $\pi_1 (\Ker(F, \phi)) = 0$ and $\pi_0(\Coker(G,
\phi)) = 0$.

\begin{df}
  A 2-\emph{exact sequence of complexes of Picard groupoids} is a diagram
\begin{equation}
\label{short-2-exact-complex}
\xymatrix{ 
0 \ar[r] &\A^{\bullet} \ar[r]_{F^\bullet} \rruppertwocell<10>^0{\omit}  &\B^{\bullet}
\ar[r]_{G^\bullet} & \C^{\bullet} \lltwocell<\omit>{<3>\phi^\bullet} \ar[r] &0,
 \\
}
\end{equation}
where $F^\bullet$ and $G^\bullet$ are $1$-morphisms and $\phi^\bullet$
is a $2$-morphism in $\2ch$, such that in every degree, the above
diagram in $\2ch$, reduces to a short $2$-exact sequence of Picard
groupoids.
\end{df}

The following example of a short $2$-exact sequence is of particular
interest and would be referenced frequently.
\begin{ex}
\label{inclusion-of-simplicial-sets}
Let $X_{\bullet}$ be a simplicial set and let $Y_{\bullet} \subset
X_{\bullet}$ be a simplicial subset. Then for any Picard groupoid
$\A$, there is a morphism $i:\C^{\bullet}(Y_{\bullet}; \A) \to
\C^{\bullet} (X_{\bullet}; \A)$ of (cochain) complexes of Picard
groupoids. This morphism determines a short $2$-exact sequence of
complexes of Picard groupoids:
\begin{equation}
\label{cohomology-short-2-exact-seq}
\xymatrix{ 0 \ar[r] &{\C^{\bullet}(X_{\bullet}, Y_{\bullet}; \A)}
  \ar[r] \rruppertwocell<10>^0{\omit}
  &{\C^{\bullet}(X_{\bullet}; \A)} \ar[r]_{i}
  &{\C^{\bullet}(Y_{\bullet}; \A)} \lltwocell<\omit>{<3>\pi
    \ \ } \ar[r] &0.  \\
}
\end{equation}
The inclusion of simplicial sets induces another morphism of (chain)
complexes of Picard groupoids, $i: \C_{\bullet}(Y_{\bullet}; \A) \to
\C_{\bullet} (X_{\bullet}; \A)$.  This morphism determines a short
$2$-exact sequence of (chain) complexes of Picard groupoids:
\begin{equation}
\label{homology-short-2-exact-seq}
\xymatrix{ 0 \ar[r] &{\C_{\bullet}(Y_{\bullet}; \A)} \ar[r]_{i}
  \rruppertwocell<10>^0{\omit} &{\C_{\bullet}(X_{\bullet}; \A)}
  \ar[r] &{\C_{\bullet}(X_{\bullet}, Y_{\bullet}; \A)}
  \lltwocell<\omit>{<3>\pi \ \ } \ar[r] &0.
  \\
}
\end{equation}
\end{ex}

A short 2-exact sequence of complexes \eqref{short-2-exact-complex}
has an associated \emph{long $2$-exact sequence} of cohomology
\begin{equation}
\label{long-2-exact-sequence}
\xymatrix@C=15pt{ \dots \to
  H^n(\A^{\bullet}) \ar[r]_{\;\;\;\;\;\;  H^n(F)} \rruppertwocell<10>^0{\omit}
  & H^n(\B^{\bullet}) \ar[r]^{H^n(G)} \rrlowertwocell<-10>_0{<3>\ \ \Sigma^n } & H^n(\C^{\bullet})
  \lltwocell<\omit>{<3>H^n(\phi) \ \ \ \ \ } \ar[r]^{\partial^n} \rruppertwocell<10>^0{\omit} &
  H^{n+1}(\A^{\bullet}) \ar[r]_{\!\!\!\!\! H^{n+1}(F)} & \lltwocell<\omit>{<3>\Psi^n \ \ } H^{n+1}(\B^{\bullet}) \to
  \dots }
\end{equation}
We will briefly outline the construction of the 1-morphism
$\partial^n$ and the $2$-morphism $\Psi^n$ here. For a more
elaborate description of the various components of this long exact
sequence, we refer the interested reader to Section 4 of
\cite{dr-mm-v}. For an outline we will refer to the following diagram:
\begin{equation*}
  \xymatrix@C=22pt@!C{ \!\!\!\!\!\!\!\!\!\!\!\!\!\! \dots  \to
    \A^{n} \ar[d]_{d^n_{\A}} \ar[r]^{f^n} \rruppertwocell<10>^0{\omit}
    & \B^{n} \ar[d]_{d^n_{\B}} \ar[r(0.8)]^{g^n} \ar@{<=}_{\lambda^{n}}(20,-5)*{}
    ;(16,-9)*{}  & \;\;\;\;\;\;\;\;\;\;\;  \C^{n} \ar[d]^{d^n_{\C}} \lltwocell<\omit>{<3>\!\! \phi^{n}}
    \ar@{<=}^{\mu^{n}}(57,-5)*{};(53,-9)*{} \;\;\; \to  \dots
    \\
    \!\!\!\!\!\!\!\!\! \dots \to
    \A^{n+1} \ar[r]_{f^{n+1}} \rrlowertwocell<-10>_0{<2.7>
      \!\!\!
      \!\!\!\!\!\!\!\!\!\!\!\!\!\!\!\!\!\!\!\!\!\!
\phi^{n+1}}
    & \B^{n+1} \ar[r(0.8)]_{g^{n+1}}  & \;\;\;\;\;\;\;\;\;\;  \C^{n+1} \to  \dots }
\end{equation*}
Let $(C_n, c_n: d_{\C}^n(C_n) \to 0)$ be an object in $Ker(d_{\C}^n)$;
since $\pi_1(\Coker(g_n,\phi^{n})) = 0$, there is a $B_n \in \B^n$ and
$i: g^n(B_n) \to C_n$. Since the following pair
\[
 (d_{\B}^n(B_n), c_n \circ d_{\C}^n(i) \circ \mu_n(B_n):g^{n+1}(d_{\B}^n(B_n)) \to d_{\C}^n(g^n(B_n))
\to d_{\C}^n(C_n) \to 0)
\]
is an object of $Ker(g^{n+1})$ and the factorization of
$f^{n+1}:\A^{n+1} \to \B^{n+1}$ through $Ker(g^{n+1})$ is an
equivalence, there are $A_{n+1} \in \A^{n+1}$ and $j: f^{n+1}(A_{n+1})
\to d^n_{\B}(B_n)$ such that
\[
g^{n+1}(j) \circ c_n \circ d_{\C}^n(i) \circ \mu_n(B_n)  =
\phi^{n+1}(A_{n+1}).
\]
We now need an arrow $a_{n+1}:d^{n+1}_{\A} \to 0$. Since the
factorization ${f'}^{n+2}$ of $f^{n+2}$ through $Ker(g^{n+2})$ is an
equivalence of categories, it is enough to find an arrow
${f'}^{n+2}(d^{n+1}_{\A}(A_{n+1})) \to {f'}^{n+2}(0)$. This is given
by the following:
\[
 \chi_{\B}(B_n) \circ d^{n+1}_\B(j) \circ \lambda^{n+1}(A_{n+1}): f^{n+2}
(d^{n+1}_\A(A_{n+1})) \to 0 \cong f^{n+2}(0).
\]
We put $\partial^{n}(C_n, c_n) := (A_{n+1}, a_{n+1})$. This is an
object of $H^{n+1}{\A^{\bullet}}$: the condition $d^{n+1}_\A(a_{n+1})
= \chi_\A (A_{n+1})$ can be easily checked by applying the faithful
functor $f^{n+3}$.  The arrow function of the functor $\partial^n$ has
a much more elaborate description and moreover it is not used in the
construction of our theory. We will refer the interested reader to
\cite{dr-mm-v}.

Before we can describe a construction of the $2$-morphism $\Psi^n$, we
need another description of $H^{n}(\C^\bullet)$. Since $(f^n, \phi^n,
g^n)$ is a $2$-short exact sequence, $\C^n$ is equivalent to the
cokernel of $f^n$, we get the following alternative description of
$H^n(\C^\bullet)$. An object is a pair
\[
(B_n \in \B^n, [A_{n+1} \in \A_{n+1}, a_{n+1}: d^n_{\B}(B_n) \to
f^{n+1}(A_{n+1})]),
\]
where $[A_{n+1}, a_{n+1}] \in \Mor_{\Coker (f^{n+1},\id_0)}(d^n_{\B}(B_n), 0)$,
such that there exists an arrow $t^{n+2}:d^n_{\A}(A_{n+1}) \to 0$
making the following diagram commutative
\[
\begin{CD}
  d^{n+1}_{\B}(d^n_{\B}(B_n)) @>{d^{n+1}_{\B}(a_{n+1})}>>
  d^{n+1}_{\B}(f^{n+1}(A_{n+1}))\\
  @V{\chi_\B(B_n)}VV @VV{(\lambda^{n+1})^{-1}(A_{n+1})}V \\
  0 @<<{f^{n+2}(t^{n+2})}< f^{n+2}(d^{n+1}_{\A}(A_{n+1})) .
\end{CD}
\]
Note that $t^{n+2}$ is necessarily unique because $f^{n+2}$ is faithful.
Now we begin the construction of $\Psi^n$, given an object
\[
(B_n \in \B^n, [A_{n+1} \in \A^{n+1}, a_{n+1}: d^n_{\B}(B_n) \to
f^{n+1}(A_{n+1})]),
\]
in $H^n(C^\bullet)$, we apply $\partial^n$ and $H^{n+1}(f)$ and
obtain the following object of $H^{n+1}(\B^\bullet)$:
{\small
\[
 (f^{n+1}(A_{n+1}), f^{n+2}(t^{n+2}) \circ \lambda^{-1}_{n+1}
 (A_{n+1}):d^{n+1}_{\B}(f^{n+1}(A_{n+1})) \to f^{n+2}(d^{n+1}_{\A}
 (A_{n+1})) \to 0).
\]
}This object is naturally isomorphic to the unit of the addition
$0 \in H^{n+1}(\B^\bullet)$ via the following morphism which we
take as the definition of $\Psi^n$ on the object
$(B_n, [A_{n+1}, a_{n+1}])$
\[
 \Psi^n(B_n, [A_{n+1}, a_{n+1}]) := [B_n \in \B^n, a_{n+1}^{-1}:
 f^{n+1}(A_{n+1}) \to d_\B^n(B_n)].
\]
The following example describes the images under the morphism
$\partial^n$ and the natural transformation $\Psi^n$ of an object in
degree $n$ of the cohomology sequence associated to the $2$-short
exact sequence \eqref{homology-short-2-exact-seq}.

\begin{ex}
\begin{sloppypar}
  Let $X$ be a compact finite-dimensional manifold with boundary
  $\partial X$.  We denote by $H_n(X, \partial X;\L)$ the $n$th
  homology Picard groupoid of the chain complex
  $C_{\bullet}(X_\bullet, \partial X_{\bullet}; \L)$. Let $(X_{n},
  [X'_{n-1},x'_{n-1}]) \in \Ob \ H_n(X, \partial X;\L)$, where $X_n
  \in \Ob \ C_n(X_\bullet; \L)$ and the morphism
  $[X'_{n-1},x'_{n-1}]:d(X_n) \to 0$ in $\Mor \ C_{n-1}(X_\bullet;
  \partial X_\bullet; \linebreak[0] \L)$ consists of an object
  $X'_{n-1} \in C_{n-1} (\partial X_\bullet;\L)$ and a morphism
  $x'_{n-1}:d^n(X_n) \to X'_n$ $\in \Mor \ C_{n-1}(X_\bullet;
  \L)$. The coboundary $\partial_n (X_{n}, [X'_{n-1},x'_{n-1}])$ is
  the pair $(X'_{n-1}, \linebreak[0] x'_{n-1}) \in \Ob \
  H_{n-1}(\partial X;\L)$. The natural transformation $\Psi_n(X_{n},
  [X'_{n-1},x'_{n-1}])$ is a morphism in $ \Hom_{H_{n-1}
    (X;\L)}((X'_{n-1}, x'_{n-1}), 0)$ given by the equivalence class
  $[X_n,(x'_{n-1})^{-1}]$. Thus every object in $H_n(X, \partial
  X;\L)$ produces a morphism in $H_{n-1}(X ;\L)$.
\end{sloppypar}
\end{ex}

\subsection{The Cap Product}
In this section we develop a cap product between cohomology with
coefficients in a Picard groupoid and homology with coefficients in
the Picard groupoid $\Z[0]$:
\[
\cap : H_\bullet (X_\bullet, \Z[0]) \otimes H^\bullet(X_\bullet, \A)
\rightarrow H_\bullet (X_\bullet, \A).
\]
In order to do that, we will define a chain map i.e a morphism in $\2ch$
\[
H_\bullet (X_\bullet, \Z[0]) \to [H^\bullet(X_\bullet, \A), H_\bullet (X_\bullet, \A)].
\]
We start by defining the following chain map
\begin{equation}
\label{chain-map-cap-product} 
\cap^{ch}:C_\bullet (X_\bullet, \Z[0]) \to [C^{\bullet}(X_\bullet, \A) , C_\bullet (X_\bullet, \A)]
\end{equation}
where the right hand side is the chain complex $[C^\bullet (X_\bullet,
\A),C_\bullet (X_\bullet, \A)]_{\bullet}$ defined in Appendix 2. We
define the map in degree $p$ as follows: On objects the chain map is
given by
\[
 \sigma_{q} \mapsto \underset{q \ge p}\prod F_{q},
\]
where $F_{q} \in \Mor(\Pic)$ is defined on objects by
\[
 F_{q}: \alpha \mapsto \alpha(d_{f}^{p}(\sigma_{q}))d_{l}^{q-p}(\sigma_{q}),
\]
where $d_{f}^{p}$ and $d_{l}^{q-p}$ are the restrictions of $d:C_{q+1}
(X_\bullet, \Z[0]) \rightarrow C_{q} (X_\bullet, \Z[0]) $ to the
simplex determined by the first $p+1$ vertices and the last $q-p+1$
vertices of $\sigma_{q}$ respectively.  On morphisms, $F_{q}$ given by
\[
  F_{q}:\lbrace\lbrace\alpha
    \to \beta \rbrace \mapsto \lbrace \alpha(d_{f}^{p}(\sigma_{q}))d_{l}^{q-p}(\sigma_{q})
    \to  \beta(d_{f}^{p}(\sigma_{q}))d_{l}^{q-p}(\sigma_{q}) \rbrace
\]
The map on the right side is determined by the natural
transformation $\alpha \to \beta$.
This chain map induces a map on homology
\begin{equation}
 H_{\bullet}(X_\bullet, \Z[0]) \to H_{\bullet}( [C^\bullet (X_\bullet, \A), C_\bullet (X_\bullet, \A)]).
\end{equation}

Composition with the following obvious morphism gives us the desired chain map
\begin{equation}
\label{homology-to-homology-of-internal-hom}
 H_{\bullet}( [C^\bullet (X_\bullet, \A), C_\bullet (X_\bullet, \A)]) \to [H^\bullet (X_\bullet, \A), H_\bullet (X_\bullet, \A)].
\end{equation}

\subsection{Relative cap product}

We now construct a relative version of the cap product. The $2$-functor
$[C^{\bullet}(X_{\bullet}; \A), -]:\2ch \to \2ch$ and the chain map
\eqref{chain-map-cap-product}, determine a composite $1$-morphism
and a $2$-morphism $\phi_{\bullet}$ in $\2ch$
\[ 
\xymatrix@C=1pt@!C{
C_{\bullet}(Y_{\bullet}; \Z[0]) \ar[r]^{i_\bullet} \rruppertwocell<10>^0{\omit} 
& C_{\bullet}(X_{\bullet}; \Z[0]) \ar[r]^{\cap^{ch}_{rel} \ \ \ \ \ \ } &[C^{\bullet}(X_{\bullet}; \A),
C_{\bullet}(X_{\bullet}, Y_{\bullet}; \A)] \lltwocell<\omit>{<3>\phi_{\bullet} \ \ }. \\
}
\]
In order to define the $1$-morphism $\cap^{ch}_{rel}$ and the $2$-morphism $\phi_\bullet$
in the above diagram, we need to define a
restriction of the chain map \ref{chain-map-cap-product}. The image of the restriction
of this chain map to $i_{\bullet}(C_{\bullet}(Y_{\bullet}; \Z[0]))$ is
contained in the $2$-(chain) complex $[C^{\bullet}(X_{\bullet}; \A),
C_{\bullet}(Y_{\bullet}; \A)]$, this determines the following commutative diagram
\[ 
\xymatrix@C=65pt@R=35pt{
i_\bullet(C_{\bullet}(Y_{\bullet}; \Z[0]))
\ar[r]^{\cap^{ch}|_{i_\bullet(C_{\bullet}(Y_{\bullet}; \Z[0]))} \ \ } \ar[rd]_{\cap^{ch}_{Y_\bullet}}
&[C^{\bullet}(X_{\bullet}; \A),C_{\bullet}(X_{\bullet}; \A)] \\
&[C^{\bullet}(X_{\bullet}; \A),C_{\bullet}(Y_{\bullet}; \A)] \ar[u]_{[C^{\bullet}(X_{\bullet}; \A),i_\bullet]} 
}
\]
The following composite chain map will be called the \emph{restricted relative cap product chain map}.
\begin{equation}
 \label{restricted-cap-product}
\cap^{ch}_{res}:C_{\bullet}(Y_{\bullet}; \Z[0]) \overset{i_\bullet}
\to i_\bullet(C_{\bullet} (Y_{\bullet}; \Z[0])) \overset{\cap^{ch}_{Y_\bullet}}
\to [C^{\bullet}(X_{\bullet}; \A), C_{\bullet} (Y_{\bullet}; \A)]
\end{equation}

The $2$-morphism $\phi_{\bullet}$ is the composition $\cap^{ch}_{res}
\circ ([C^{\bullet}(X_{\bullet}; \A),\pi_{\bullet}^{\A})])$ as
described in the following diagram
\[ 
\xymatrix@C=10pt{
C_{\bullet}(Y_{\bullet}; \Z[0]) \rruppertwocell<12>^0{\omit}
\ar[d]_{\cap^{ch}_{res}} \ar[r]^{i_\bullet} &C_{\bullet}(X_{\bullet}; \Z[0])
\ar[d]_{\cap^{ch}} \ar[rd]_{\cap^{ch}_{rel}} \ar[r]^{p_\bullet} &C_{\bullet}(X_{\bullet}, Y_{\bullet};
\Z[0]) \ar@{=>}^{\ \ \lambda_\bullet}(72,-4)*{} ;(67,-6)*{}  \ar@{-->}[d]^{u}
 \lltwocell<\omit>{<3> \pi_{\bullet} }.\\
[C^{\bullet}(X_{\bullet}; \A), C_{\bullet}(Y_{\bullet}; \A] \ar[r] \rrlowertwocell
<-12>_0{<3> \ \ \ \ \ \ \ \ \ \ \ \ \ \ \ [C^{\bullet}(X_{\bullet}; \A),\pi_{\bullet}^{\A})] }
& [C^{\bullet}(X_{\bullet}; \A),C_{\bullet}(X_{\bullet}; \A) \ar[r]
 &[C^{\bullet}(X_{\bullet}; \A), C_{\bullet}(X_{\bullet}, Y_{\bullet}; \A)] \\
}
\]
where $\pi_{\bullet}^{\A}$ is the following $2$-morphism
\[ 
\xymatrix@C=26pt{
C_{\bullet}(Y_{\bullet}; \A) \ar[r]^{i_{\bullet}} \rruppertwocell<12>^0{\omit} 
& C_{\bullet}(X_{\bullet}; \A) \ar[r]_{p^\A_\bullet } &
C_{\bullet}(X_{\bullet}, Y_{\bullet}; \A) \lltwocell<\omit>{<3>\pi_{\bullet}^{\A} \ \ }. \\
}
\]

The universality of the cokernel determines a unique pair consisting of a $1$-morphism in $\2ch$
\[
u:C_{\bullet}(X_{\bullet}, Y_{\bullet};\Z[0]) \to [C^{\bullet}(X_{\bullet}; \A),
C_{\bullet}(X_{\bullet}, Y_{\bullet}; \A)],
\]
and a $2$-morphism in $\2ch$, $\lambda_\bullet: u \circ p_\bullet \Rightarrow \cap^{ch}_{rel}$
such that the following diagram commutes
\[ 
\xymatrix@C=20pt{
u \circ p_\bullet \circ i_\bullet \ar@{=>}[r]^{\ \ \ \ \ \ u \cdot \pi_\bullet}
\ar@{=>}[d]_{\lambda_\bullet \cdot i_\bullet} & u \circ 0 \ar@{=>}[d] \\
\cap^{ch}_{rel} \circ i_\bullet \ar@{=>}[r]_{\phi_\bullet} &0
}
\]
For more details on the universality of this cokernel we refer the
interested reader to \cite{kas-vit}.  This unique $1$-morphism, $u$,
induces the following $1$-morphism on passing to homology
\begin{equation}
 H_{\bullet}(X_\bullet, Y_\bullet, \Z[0]) \to H_{\bullet}( [C^\bullet (X_\bullet, \A),
C_\bullet (X_\bullet, Y_\bullet, \A)]).
\end{equation}
Composition with the following chain map
\begin{equation}
\label{Relative-homology-to-homology-of-internal-hom}
 H_{\bullet}( [C^\bullet (X_\bullet, \A), C_\bullet (X_\bullet, Y_\bullet, \A)]) \to
 [H^\bullet (X_\bullet, \A), H_\bullet (X_\bullet, Y_\bullet, \A)].
\end{equation}
and the adjointness of the tensor product gives us the desired chain map
\begin{equation}
\label{relative-cap-product}
\cap :  H_\bullet (X_\bullet, Y_\bullet, \Z[0])  \otimes H^\bullet(X_\bullet, \A)
\rightarrow H_\bullet (X_\bullet, Y_\bullet, \A).
\end{equation}
which we will call the relative cap product.

As in the classical case, the boundary map in homology is natural with
respect to relative cap product, in a sense made precise below.
\begin{prop}
\label{cap-del}
  The following diagram of Picard groupoids is commutative up to
  natural isomorphism for $p-1 \ge q \ge 0$:
\begin{equation*}
  \xymatrix{
    H_{p} (X_\bullet, Y_\bullet; \Z[0]) \otimes H^q(X_\bullet, \A)
    \ar[r]^-{\cap} \ar[d]_{\del \otimes i^*}& H_{p-q} (X_\bullet, Y_\bullet; \L)
    \ar[d]^{\partial} \ar@{=>}(30,-5)*{};(20,-9)*{}
    \\
    H_{p-1} (Y_\bullet; \Z[0]) \otimes H^q(Y_\bullet, \A)     \ar[r]^-{\cap}
    & H_{p-q-1} (Y_\bullet; \L)
  }.
\end{equation*}
\end{prop}

\section{Hermitian line gerbes}
\label{gerbes}

In this section we describe geometric objects which we call
\emph{$($flat$)$ hermitian line $n$-gerbes}. Then we give an example
describing a \ChGG\ over the simplicial set $BG$, where $G$ is a
(discrete) group. We move on to describe certain geometic objects over
a topological space $X$ which are classified by the \Ch\ cohomology of
$X$ with coefficients in $U(1)$ and which we call \emph{\ChG s} over
$X$. We describe these in two ways. For the first description, we
define a category $\cu$, associated to an open cover of $X$ and show
that hermitian line $0$-co\-cy\-cles on the simplicial set $N(\cu)$
represent (flat) \ChZG s.  Our second description is that a flat
\ChZG\ can be represented by a functor from the first fundamental
groupoid of $X$ into the Picard groupoid of hermitian lines
$\L$. Finally, we move on to describe higher hermitian line gerbes
over $X$.

\begin{df}
A \emph{hermitian line $n$-gerbe on} a simplicial set $X_\bullet$ is
an an $n$-cocycle on the simplicial set $X_\bullet$ with values in $\L$
i.e. an object $K \in \Ob H^n (X_\bullet, \L)$ of degree $n$ cohomology
Picard groupoid of $X_\bullet$ with coefficients in the Picard groupoid
$\L$ of hermitian lines.
\end{df}

\begin{rem}
 A $n$-gerbe should be properly defined as a $0$-cocycle with
 coefficients in an appropriate Picard $(n+1)$-groupoid
 but this would be out of scope of this paper.
\end{rem}

The following example shows that a $2$-gerbe on $BG$, where
$G$ is a finite group, is exactly the same as a $2$-cocycle
with values in hermitian lines as defined in \cite{fhlt}.

\begin{ex}
 Any group $G$ can be viewed as a category with a single object.
 The simplicial set $BG$ is the nerve of this category.
 An object of $H^2(BG;\L)$ consists of a pair $(\beta, \phi:d\beta \to 0)$,
 where $\beta \in \Ob C^2(BG;\L)$ and $\phi$ is an arrow in $C^3(BG;\L)$
 and $d$ is the differential of the $2$-complex $C^\bullet(BG;\L)$.
 $\beta$ is a set function whose domain is the underlying set
 of $G \times G$ and codomain is $\Ob \L$ i.e. it assigns to
 each pair $(g_1, g_2) \in G \times G$ a hermitian line
 $l_{g_2,g_1} \in \Ob \L$. The arrow $\phi$ gives, for
 every triple $(g_1, g_2, g_3) \in G \times G \times G$,
 a following isomorphism in $\L$
 \[
  t_{g_3, g_2, g_1}:l_{g_3, g_2} - l_{g_3,g_2g_1}
  +l_{g_3g_2, g_1} - l_{g_2, g_1} \to \mathbb{C}.
 \]
 The morphism $d(\phi):d^2(\beta) \to 0$
 gives, for every quadruple $(g_1, g_2, g_3, g_4) \in G \times G
 \times G \times G$, the following isomorphism
 \[
 t_{g_4,g_3,g_2} - t_{g_4,g_3,g_{21}} + t_{g_4,g_{32},g_1}
 -t_{g_{43,g_2,g_1}} + t_{g_3,g_2,g_1},
 \]
 which is the canonical isomorphism $\chi_\beta((g_1, g_2, g_3, g_4)):
 d^2\beta((g_1, g_2, g_3, g_4)) \to \mathbb{C}$.
 
\end{ex}

By the definition above, a flat \ChZG\ on a simplicial set $X_\bullet$
is just an object of the Picard groupoid $H^0(X_\bullet; \L)$. Thus,
given a topological space $X$, we may look at \ChZG s over $X$ in two
ways: associating simplicial sets $\Sing_\bullet X$ and $N(\cu)$ to
$X$, where $N(\cu)$ is the simplicial sets obtained by taking the
nerve of a category associated to the cover of $X$, $\cu$, which we
now define:
\begin{df}
  \label{category-associated-to-cover}
  We define $\cu$ to be a category whose object set is
 the collection $\mathfrak{U}_I = \lbrace U_i: i \in I \rbrace\ $,
 which is a chosen open cover of the topological space $X$. If the
 set $U_{i,j} \neq \emptyset$, then $\Hom_{\cu}(U_i, U_j)
 = \lbrace U_{i,j} \rbrace$, otherwise the set
 $\Hom_{\cu}(U_i, U_j) = \emptyset$. Composition in $\cu$
 is defined as follows: $U_{i,j} \circ U_{j,k} := U_{i,j,k}
 := U_{i,j} \cap U_{j,k}$. $id_{U_{i}} := U_{ii}$. The source
 of an arrow $U_{i,j}$ is $U_i$ and its target is $U_j$. 
 \end{df}
 This leads to two interpretations of flat \ChZG s: Definitions
 \ref{0-gerbe} and \ref{c-0-gerbe} below.
 \begin{df}
\label{0-gerbe}
A \emph{flat \ChZG} over $X$, $\mathcal{G}^0(\Lambda, \theta)$, is
defined by the following data
 \begin{enumerate}
  \item [1.] A function $\Lambda_0:(\Sing_\bullet X)_0 \to \Ob(\L)$,
  i.e. an assignment of a hermitian line to each point of $X$.
\item [2.] A function $\Lambda_1:(\Sing_\bullet X)_1 \to \Mor(\L)$
  which assigns to each $f \in (\Sing_\bullet X)_1$, a linear isometry
  $\Lambda_1(f):\Lambda_0\partial_1(f) \to \Lambda_0\partial_0(f)$ in
  $\L$ such that for all $f, g \in (\Sing_\bullet X)_1$ satisfying
  $\partial_1(f) = \partial_0(g)$, $\Lambda_1(g \circ f) = \Lambda(g)
  \circ \Lambda(f)$.
 \end{enumerate}
 This data is subject to the following condition. For each
 $n \ge 2$, there exists a function $\Lambda_n:(Sing_\bullet X)_n \to \Mor(\L)$
 such that for all $\sigma_n \in (Sing_\bullet X)_n$,
 $\Lambda_n(\sigma_n) = \Lambda_{n-1}(\partial_0\sigma_n) \circ 
 \Lambda_{n-1}(\partial_1\sigma_n) \circ \cdots \circ \Lambda_{n-1}
 (\partial_{n-1}\sigma_n)$.
 \end{df}
 \begin{rem}
  The above definition assigns a hermitian line to each point of $X$.
  Further, two homotopic paths in $X$, relative to endpoints,
  are assigned the same linear isometry. In other words the above data
  is equivalent to defining a functor from the first fundamental groupoid
  of the space $X$, $\Pi_1(X)$, to $\L$.
 \end{rem}
 \begin{df}
\label{c-0-gerbe}
A \emph{flat \ChZG} over $X$, $\mathcal{G}^0(\Lambda, \theta)$, is
defined by the following data
 \begin{enumerate}
  \item [1.] A constant hermitian line bundle $\Lambda_i$
  over every open set $U_{i}$ for all $i \in I$.
  
  \item [2.] For each ordered pair of distinct indices
  $(i, j) \in I \times I $, a constant, non-zero section
  \[
   \theta_{i,j} \in \Gamma(U_{i,j}; \Lambda_i \otimes
   \Lambda_j)
  \]
 \end{enumerate}
 This data is subject to a cocycle condition, on $U_{i,j,k}$
 which we denote by $\delta \theta \Rightarrow 0$.
 The cocycle condition is that over any three fold
 intersections $U_{i,j,k}$, we can tensor the three
 sections of the coboundary to give a trivialization
 of the following hermitian line bundle
 \begin{equation*}
  (\Lambda_i \otimes \Lambda_j) \bigotimes
  (\Lambda_i \otimes \Lambda_k)^{-1} \bigotimes
  (\Lambda_j \otimes \Lambda_k).
 \end{equation*}
 over $U_{i,j,k}$. Notice that the above hermitian line bundle
 is canonically trivial, so the cocycle condition
 is the requirement that the following
 \[
  \theta_{i,j} - \theta_{i,k} + \theta_{j,k}
 \]
 be the canonical section of this trivial hermitian line bundle
 over $U_{i,j,k}$.

 \end{df}
\begin{rem}
 Each point $x \in X$ has a neighborhood $U_i$ such that the
 hermitian line bundle $\Lambda_i$ is isomorphic to the
 trivial hermitian line bundle $U_i \times \mathbb{C}$.
 Further, the specification of constant, non-zero section
 $\theta_{i,j}$ is the same as specifying a hermitian line
 bundle isomorphism $g_{i,j}:\Lambda_i|_{U_{i,j}} \to \Lambda_j|_{U_{i,j}}$,
 which restricts to the same linear isometry on every fiber.
 These two observations along with the data in the definition
 above are sufficient to construct a (flat) hermitian line bundle
 over the space $X$.
 \end{rem}

 Now we move on to define higher hermitian line gerbes. Our definition
 of a \ChG\ closely follows the definition of a ``1-gerb'' developed
 in \cite{DSC}.

\begin{df}
  A \emph{\ChG} over $X$, $\mathcal{G}^1(\Lambda, \theta)$, is defined
  by the following data
 \begin{enumerate}
  \item [1.] A constant hermitian line bundle $\Lambda_i^j$
  over the intersection $U_{i,j}$ for every ordered pair
  $(i, j) \in I \times I$ and $i \ne j$, such that
  $\Lambda_i^j$ and $\Lambda_j^i$ are dual to each other.
  
  \item [2.] For each ordered triple of distinct indices
  $(i, j, k) \in I \times I \times I$, a nowhere zero section
  \[
   \theta_{i,j,k} \in \Gamma(U_{i,j,k}; \Lambda_i^j \otimes
   \Lambda_j^k \otimes \Lambda_k^i)
  \]
  such that the sections of reorderings of triples
  $(i, j, k)$ are related in the natural way.
 \end{enumerate}
 This data is subject to a cocycle condition, on $U_{i,j,k,l}$
 which we denote by $\delta \theta \Rightarrow 0$.
 The cocycle condition is that over any four fold
 intersections $U_{i,j,k,l}$, we can tensor the four
 sections of the coboundary to give a trivialization
 of the following hermitian line bundle
 \begin{equation}
 \label{Delta-Lambda}
  (\Lambda_i^j \otimes \Lambda_j^k \otimes \Lambda_k^i) \bigotimes
  (\Lambda_i^j \otimes \Lambda_j^l \otimes \Lambda_l^i)^{-1} \bigotimes
  (\Lambda_i^k \otimes \Lambda_k^l \otimes \Lambda_l^i) \bigotimes
  (\Lambda_j^k \otimes \Lambda_k^l \otimes \Lambda_l^j)^{-1}.
 \end{equation}
 over $U_{i,j,k,l}$. Notice that the above hermitian line bundle
 is canonically trivial, so the cocycle condition
 is the requirement that the following
 \[
  \theta_{i,j,k} - \theta_{i,j,l} + \theta_{i, k, l}
  -\theta_{j,k,l}
 \]
 be the canonical section of this trivial hermitian line bundle
 over $U_{i,j,k,l}$.
\end{df}
 
The tensor product of two \ChG s is obtained by tensoring line bundles
and sections in an obvious way.
 
Let $(\alpha, \phi) \in \Ob H^1(N(\cu;\L)$.  To each $U_{i,j}$, the
cochain $\alpha$ assigns a hermitian line $l_i^j$ and the morphism
$\phi$ specifies a linear isometry for each $U_{i,j,k}$
 \[
  \phi(U_{i,j,k}):l_i^j - l_k^i + l_j^k \to \mathbb{C}.
 \]
  Equivalently, the specification of this linear isometry
 is the specification of a constant function
 $t_{i, j, k}:U_{i,j,k} \to U(1)$. In other words
 \[
  t_{i,j,k}(x) = \phi(U_{i,j,k}),
 \]
 $\forall x \in U_{i,j,k}$. Put constant
 hermitian line bundles $\Lambda_i^j = U_{i,j}
 \times l_i^j$ over each $U_{i,j}$. Then $t_{i,j,k}$
 gives a trivialization of the coboundary line bundle
 $\Lambda_i^j \otimes \Lambda_j^k \otimes \Lambda_k^i$.
 We define the section
 \[
  \theta_{i,j,k}(x) = t_{i,j,k}^{-1}(x, e_1).
 \]
 The morphism
  \[
 dt_{i,j,k} = d\phi(U_{i,j,k,l}) = t_{i,j,k} -t_{i,j,l}
 + t_{i,k,l} - t_{j,k,l}
 \]
 gives a trivialization of the line bundle \eqref{Delta-Lambda} over
 $U_{i,j,k,l}$.  The following section corresponds to the above
 trivialization of the the hermitian line bundle \eqref{Delta-Lambda}
 \[
  \theta_{i,j,k} - \theta_{i,j,l} + \theta_{i,k,l}
  -\theta_{j,k,l}.
 \]
 Clearly this is the canonical section.
%

\begin{sloppypar}
  Conversely, a \ChG\ over $X$ defines a $1$-cocycle in $H^1(N(\cu;
  \L))$.  We leave the easy verification of this fact as an excercise
  for the reader.
\end{sloppypar}
 
 \begin{df}
  \label{equivalence-of-locally-trivialized-gerbes}
  Let $\mathcal{G}^1(\Lambda, \theta)$ and $\mathcal{H}^1(\Upsilon,
  \eta)$ be two \ChG s over $X$ and let $(g, \phi)$ and $(h, \psi)$
  the two $1$-cocycles in $H^1(N(\cu;\L))$ determined by them, then
  the the two gerbes $\mathcal{G}^1(\Lambda, \theta)$ and
  $\mathcal{H}^1(\Upsilon, \eta)$ are equivalent if there exists a
  morphism $(g, \phi) \to (h, \psi)$ in $H^1(N(\cu);\L)$.
  
 \end{df}
 If $\mathcal{G}^1(\Lambda, \theta)$ and $\mathcal{H}^1
 (\Upsilon, \eta)$ are equivalent, then there are
 hermitian line bundle isomorphisms
  \[
   \Lambda_i^j \cong \Upsilon_i^j,   
  \]
  over each $U_{i,j}$, such that the isomorphisms induce
  a mapping
  \[
   \theta_{i,j,k} \mapsto \eta_{i,j,k}.
  \]
 
 \begin{df}
 \label{global-trivialization-of-cech-1-gerbe}
 A \ChG\ $\mathcal{G}^1(\Lambda, \theta)$ is \emph{globally
   trivialized} by displaying a basis $\lambda_i^j$ for each line
 bundle $\Lambda_i^j$ such that on each $U_{i,j,k}$, we can express
 the sections on three fold intersections, in terms of coordinates
 specifed specified by the data and the ring
 $C^\infty(U_{i,j,k};U(1))$, as follows:
  \[
   \theta_{i,j,k} = 1(x) \lambda_i^j \otimes \lambda_j^k
  \otimes \lambda_k^i,
  \]
  where $1(x) \in C^\infty(U_{i,j,k};U(1))$ is the constant
  function which assigns to each point $x \in U_{i,j,k}$, the
  identity of the group $U(1)$.
 \end{df}
 \begin{rem}
   Let $\G$ be a globally trivial \ChG\ and $(\alpha, \phi)$ be the
   \ChC\ determined by $\G$.  Then for $U_{i,j,k}$
  \[
   \phi(U_{i,j,k}):l_i^j - l_j^k + l_k^i \to
   \mathbb{C}
  \]
  is the canonical isomorphism.
 \end{rem}

 \begin{df}
 \label{trivial-cech-1-gerbe}
 A \ChG\ is \emph{trivial} if it is equivalent to the \emph{zero
   $1$-gerbe over $X$}, which is the \ChG\ determined by the cocycle
 $(0,id_0) \in H^1(N(\cu);\L)$.
 \end{df}
 
 The notion of a trivial \Ch\ hermitian line 1-gerbe can equivalently
 be defined by a geometric entity called an \emph{object}, which we
 define next.
 \begin{df}
 \label{Object}
 Given a \ChG\ $\mathcal{G}^1(\Lambda, \theta)$, an \emph{object
   compatible with} $\mathcal{G}^1$, denoted $\O(L,m)$ is specified by
 the following data
 \begin{enumerate}
  \item [1.] Constant hermitian line bundles $L_i$ over each $U_i$;
  
  \item[2.] Hermitian line bundle isomorphisms over each intersection
  $U_{i, j}$
  \[
   m_i^j:L_i \cong \Lambda_i^j \otimes L_j;
  \]
  such that the composition on three fold intersection
  \[
   L_i \longrightarrow (\Lambda_i^j \otimes \Lambda_j^k \otimes
   \Lambda_k^i) \otimes L_i
  \]
  is exactly
  \[
   (id \otimes m_k^i) \circ (id \otimes m_j^k)
   \circ m_i^j \equiv \theta_{i,j,k} \otimes id.
  \]
 Here we are abusing notation by denoting the trivialization
 determined by the section $\theta_{i,j,k}$ also by
 $\theta_{i, j, k}$.
 \end{enumerate}
 \end{df}

\begin{prop}
\begin{sloppypar}
  Let $\mathcal{G}^1$ be a \ChG\ over $X$ and let $(\alpha, \phi)$ be
  the \Ch\ hermitian line $1$-cocycle determined by $\G^1$. Then
  $\G^1$ has an object, $\O(L,m)$, compatible with it iff there is a
  \Ch\ hermitian line $0$-chain $\beta \in \Ob C^0(N(\cu);\L)$ and a
  morphism $f:(\alpha, \phi) \to (d\beta, \chi_\beta)$, in $Ker(d,
  \chi)$, such that $(\beta, f)$ is a representative of a morphism
  $[(\beta, f)]: (\alpha, \phi) \to (0, id_0)$ in $H^1(N(\cu);\L)$.
\end{sloppypar}
\end{prop}

\begin{proof}
  Let $\G^1$ be a trivial \ChG\ over $X$ as above. Then there exists a
  morphism $[(\beta, f)]: (\alpha, \phi) \to (0, id_0)$ in
  $H^1(N(\cu);\L)$.  Choose a representative $(\beta, f)$ of this
  morphism.  Now we define the constant line bundle, $L_i$, over each
  $U_i$ as follows: $L_i := U_i \times \beta(U_i)$. The linear
  isometry $f(U_{i,j}):\alpha(U_{i,j}) \to (\beta(U_i) - \beta(U_j))$
  determines a morphism of hermitian line bundles
 \[
  m_i^j:L_i \to \Lambda_i^j \otimes L_j
 \]
 over each $U_{i,j}$. The condition over three fold intersections, in
 definition \ref{Object}, follows from the equation $\chi_\beta \circ
 d(f) = \phi$.  Conversely, given an object compatible with a trivial
 \ChG\ $\G^1$, one can define the isomorphism $[(\beta, f)]:(\alpha,
 \phi) \to (0, id_0)$ in $H^1(N(\cu);\L)$.
\end{proof}
 
Finally, we are ready to define a \emph{\ChGG} over $X$.
 \begin{df}
   A \emph{\ChGG\ over $X$}, $\G^2(\G, \O, \theta)$, is defined by the
   following data
 \begin{enumerate}
 \item [1.] A \ChG\ $\G_i^j$ over the intersection $U_{i,j}$ for every
   ordered pair $(i, j) \in I \times I$ and $i \ne j$ such that
   $\G_i^j$ and $\G_j^i$ are dual to each other.
  
\item [2.] For each ordered triple of distinct indices $(i, j, k) \in
  I \times I \times I$, an object $\O_{i,j,k}$ compatible with the
  coboundary gerbe
  \[
    \G_i^j \otimes \G_j^k \otimes \G_k^i
  \]
  such that the sections of reorderings of triples
  $(i, j, k)$ are related in the natural way.
  
  \item[3.] For each ordered quadruple of distinct indices
  $(i, j, k, l) \in I \times I \times I \times I$, 
  trivializations $\theta_{i,j,k,l}$ of coboundaries
  of objects
  \[
   \O_{i,j,k} \otimes \O_{i,j,l}^{-1} \otimes
   \O_{i,k,l} \otimes \O_{j,k,l}^{-1}
  \]
  on $U_{i,j,k,l}$. Notice that each pair $(\O_{i,j,k}
  \otimes \O_{i,j,l}^{-1})$ is a line bundle over
  $U_{i,j,k,l}$ so asking for a \emph{trivialization}
  of the object is ligitimate.
 \end{enumerate}
 This data is subject to a cocycle condition, on $U_{i,j,k,l,m}$
 which we denote by $\delta \theta \Rightarrow 0$.
 The cocycle condition is that over any five fold
 intersections $U_{i,j,k,l,m}$, we can tensor the five
 sections of the coboundary objects to give a trivialization
 of the following hermitian object
 \[
  (\O_{i,j,k} \otimes \O_{i,j,l}^{-1} \otimes \O_{i,k,l}
  \otimes \O_{j,k,l}^{-1}) \bigotimes
  (\O_{i,j,k} \otimes \O_{i,j,m}^{-1} \otimes
   \O_{i,k,m} \otimes \O_{j,k,m}^{-1})^{-1}
   \]
   \[
   \bigotimes   
  (\O_{i,j,l} \otimes \O_{i,j,m}^{-1} \otimes \O_{i,l,m}
  \otimes \O_{j,l,m}^{-1}) \bigotimes
  (\O_{i,k,l} \otimes \O_{i,k,m}^{-1} \otimes \O_{i,l,m}
  \otimes \O_{k,l,m}^{-1})^{-1}
 \]
 \[
   \bigotimes   
  (\O_{j,k,l} \otimes \O_{j,k,m}^{-1} \otimes \O_{j,l,m}
  \otimes \O_{k,l,m}^{-1})
  \]
 Notice that the above object is canonically trivial,
 so the cocycle condition is that the following
 \[
  \theta_{i,j,k,l} - \theta_{i,j,k,m} + \theta_{i,j,l,m}
  -\theta_{i,k,l,m} +\theta_{j,k,l,m} 
 \]
 is the canonical section.
\end{df}
A hermitian line $2$-cocycle $(\alpha, \phi)$ represents a \ChGG\ over
$X$. We outline a construction of a \ChGG\ starting from the
$2$-cocycle $(\alpha, \phi)$. A \ChG , $\G^i_j(\Lambda, \theta)$ over
$U_{i,j}$ for every pair $(i,j) \in I \times I$, is determined by the
$2$-cocycle $(\alpha, \phi)$ as follows: Over each three-fold
intersection, $U_{i,j,k}$, a constant hermitian line bundle
$\Lambda_{i,j,k}$ is defined by $\Lambda_{i,j,k} := U_{i,j,k} \times
\alpha(U_{i,j,k})$.  On every four-fold intersection $U_{i,j,k,l}$,
the section $\theta_{i,j,k,l}$ is determined by the linear isometry
 \[
  \phi(U_{i,j,k,l}):\alpha(U_{i,j,k}) - \alpha(U_{i,k,l})
  + \alpha(U_{i,j,l}) - \alpha(U_{j,k, l}) \to \mathbb{C}.
 \]
 This section satisfies the cocycle condition $\delta \theta
 \Rightarrow 0$ over five-fold intersections, thus defining a \ChG\
 over $U_{i,j}$. Notice that the coboundary \ChG\, $\G^i_j \otimes
 \G^j_k \otimes \G^i_k$ over $U_{i,j,k}$, is trivial, therefore there
 exists an object $\O_{i,j,k}$ compatible with this trivial coboundary
 gerbe. This object $\O_{i,j,k}$ is specified by the $2$-chain $\alpha
 \in C^2(N(\cu);\L)$. Over each four-fold intersection $U_{i,j,k,l}$,
 a section $\theta_{i,j,k,l}$ of the coboundary Object $\O_{i,j,k}
 \otimes \O_{i,j,l}^{-1} \otimes \O_{i,k,l} \otimes \O_{j,k,l}^{-1}$
 is specified by the linear isometry $\phi(U_{i,j,k,l})$.


\section{Dijkgraaf-Witten theory}

We would like to recover Dijkgraaf-Witten's construction
\cite{dw} of a TQFT. In principle, we follow their
construction, using Freed-Quinn's hermitian-line incarnation
\cite{fq}, and placing it further within the framework of
cohomology with coefficients in the Picard groupoid of hermitian
lines.

\subsection{Hermitian line corresponding to a closed $n$-manifold}
\label{object}

We start with an $n$-cocycle $\alpha$ which is an object of the Picard
groupoid $H^n (BG; \L)$. For each map $f: Y \to BG$ from a closed
$n$-manifold $Y$, we take the pullback $f^* \alpha$. Consider the cap
product
\[
\cap: H^n (Y; \L) \otimes H_n (Y; \Z[0]) \to H_0 (Y; \L),
\]
which is a morphism of Picard groupoids. If we substitute the given
cocycle $\alpha$ in the first factor, we will get a morphism
\begin{equation}
  \label{capping-with-alpha}
f^* \alpha \cap - : H_n (Y; \Z[0]) \to H_0 (Y; \L).
\end{equation}
What we would like to do is to apply this morphism to the fundamental
cycle of $Y$. However, in the homology with coefficients in a Picard
groupoid, be it a discrete one, such as $\Z[0]$, no single object
represents the fundamental cycle canonically. It is rather a full
subgroupoid (not monoidal) $C_Y$ formed by all possible cycles
representing the fundamental cycle and connected by equivalence
classes of morphisms given by $n$-boundaries modulo
$(n+1)$-boundaries: a morphism $y \to y'$ is given by an $(n+1)$-chain
$x$ such that $y' = y + dx$; two morphisms $x: y \to y'$ and $x': y
\to y'$ are equivalent if there is an $(n+2)$-chain $w$ such that $x =
dw + x'$. Thus, we can restrict the above morphism
\eqref{capping-with-alpha} to this \emph{fundamental-cycle groupoid}
$C_Y$ and get a functor
\[
f^* \alpha \cap - : C_Y \to H_0 (Y; \L).
\]
If we compose this functor with the \emph{degree map}
\[
H_0 (Y; \L) \to \L
\]
which takes each linear combination $a_1 y_1 + \dots + a_k y_k$ of
points $y_1, \dots, y_k$ in $Y$ with coefficients $a_1, \dots, a_k$ in
$\L$ to the sum $a_1 + \dots + a_k$, which is an object in $\L$, we
obtain a functor
\begin{equation}
\label{cap-with-fundamental-cycle groupoid}
F: C_Y \to \L
\end{equation}
from the fundamental-cycle groupoid to the groupoid of hermitian
lines. Now we take the limit of this functor. The existence of the
limit is guaranteed by the following fact.

\begin{prop}
\label{the_line}
The functor
\[
F: C_Y \to \L,
\]
which represents the cap product of the cocycle $\alpha$ with the
fundamental-cycle group\-oid $C_Y$, has a limit,
\[
\lim_{C_Y} F,
\]
in the category $\L$ of hermitian lines.
\end{prop}

\begin{proof}
  The limit of the functor $F$ may be realized by Freed-Quinn's
  \emph{invariant-section construction}: an invariant section is a
  collection of elements in $\{s(y) \in F(y) \; | \; y \in \Ob C_Y\}$
  such that for each morphism $x: y \to y'$ in $C_Y$, we have
  $F(x)s(y) = s(y')$. The space of invariant sections is a hermitian
  line, in other words, the limit of $F$ exists, if the functor has
  \emph{no holonomy}, \emph{i.e}., $F(x) = \id$ for each automorphism
  $x: y \to y$. This is indeed the case, due to the following
  argument.

  Being an object of $H^n (Y; \L)$, the cocycle $\alpha$ is
  represented by a pair $(a,\phi)$, where $a$ is an object of $C^n (Y;
  \L)$, \emph{i.e}., a function $a: S^n(Y) \to \Ob \L$, and $\phi: da
  \to 0$ is a morphism in $C^{n+1} (Y; \L)$, \emph{i.e}., a function
  $S^{n+1} (Y) \to \Mor \L$. The functor $F: C_Y \to \L$ acts in the
  following way on objects and morphisms of the groupoid $C_Y$:
\[
  F(y)  =  a(y) \qquad \text{for $y \in \Ob C_Y$},
\]
and
\[
  F(x): a(y) \to a(y') \qquad \text{for $x \in \Mor C_Y$, $y' = y + dx$},
\]
is defined by $\phi(x): a(y') - a(y) = a (dx) = da(x) \to 0$ as a
composition of it with the structure natural transformations
\eqref{structure1}-\eqref{structure2} and their inverses.

Now suppose we have an automorphism $x: y \to y$, which in particular
means that we have a chain $x \in \Ob C_{n+1} (Y; \Z[0])$, such that
$dx = 0$. Since $H_{n+1} (Y; \Z[0])$ is trivial whenever $\dim Y = n$,
the cycle $x$ must be a boundary: $x = dw$ for some $w$. This renders
the equivalence class of the morphism $x$ to be trivial.
\end{proof}

\subsection{Linear isometry corresponding to an $(n+1)$-cobordism}
\label{morphism}

Now let $X$ be a compact $n+1$-manifold with boundary $i: \partial X
= \partial X_- \coprod \partial X_+ \subset X$. As a starting point,
we use the same $n$-cocycle $\alpha$, which is an object of the Picard
groupoid $H^n (BG; \L)$.  For any continuous function $f: X \to BG$, a
pullback of $\alpha$ along $f$ gives an $n$-cocycle $f^* \alpha$,
which is an object of the Picard groupoid $H^n (X; \L)$. Consider the
relative cap product
\[
\cap: H^n (X; \L) \otimes H_{n+1} (X, \partial X; \Z[0]) \to H_1
(X, \partial X; \L),
\]
which is a morphism of Picard groupoids. If we substitute $f^* \alpha$
in the first factor, we will get a functor
\begin{equation}
  \label{relative-capping-with-alpha}
f^* \alpha \cap - : H_{n+1} (X, \partial X; \Z[0]) \to H_1 (X, \partial X; \L).
\end{equation}
As above, we restrict this functor to the \emph{relative
  fundamental-cycle groupoid} $C_{X, \partial X}$ which is the full
subgroupoid of $H_{n+1} (X, \partial X; \Z[0])$ whose objects are all
possible relative cycles representing the relative fundamental class
of $X$. The restriction gives us a functor
\[
f^* \alpha \cap - : C_{X, \partial X} \to H_1 (X, \partial X; \L).
\]
We compose this functor first with the $2$-morphism $\Psi_1$ from (the
chain version of) the long $2$-exact sequence
\eqref{long-2-exact-sequence} and then the degree map
\begin{equation}
\label{F-map}
  \xymatrix{
  C_{X, \partial X} \ar[r]^{f^* \alpha \cap - \ \ \ \ } & H_1 (X, \partial X; \L)
   \ar[r]^{\ \ \partial_1} \rruppertwocell<10>^0{\omit}
    & H_0 (\partial X; \L) \ar[r]^{H_0(i) \ \ }  &
    H_0 (X; \L) \lltwocell<\omit>{<3>\Psi_1 \ \ } 
    \overset{deg} \to \L
  }.
\end{equation}
This diagram gives us a $2$-morphism $t: F \Rightarrow 0$, where
$F: C_{X, \partial X} \to \L$ is the composite functor in the lower
row. Consider the following diagram:
\begin{equation}
\label{big_triangle}
  \xymatrix{
H_{n+1}
    (X, \del X; \Z[0]) \ar[r]^{\ \ \ f^* \alpha \cap - } \ar[d]^{\del}& H_1 (X, \del X; \L)
    \ar[r]^{\ \ \partial_1} \rruppertwocell<10>^0{\omit} 
    & H_0 (\partial X; \L) \ar[r]^{H_0(i) \ \ }  &
    H_0 (X; \L) \lltwocell<\omit>{<3>\Psi_1 \ \ } 
    \overset{\deg} \to \L\\
H_{n}   (\del X; \Z[0]) \ar[urr]_{f|_{\del X}^* \alpha \cap -}     \ar@{=>}(9,-9)*{};(21,-4)*{}
  },
\end{equation}
where the bottom 2-morphism is comes from
Proposition~\ref{cap-del}. When we restrict the boundary 1-morphism
$\del: H_{n+1} (X, \del X; \Z[0]) \to H_{n} (\del X; \Z[0])$ to the
full subcategory $C_{X, \partial X}$, we get the following commutative
diagram:
\begin{equation*}
  \xymatrix{
    C_{X, \partial X} \ar[r] \ar[d]_{\del} \ar@{}[dr]|{\circlearrowright}& H_{n+1}
    (X, \del X; \Z[0]) \ar[d]^{\del}\\
    -C_{\partial_- X} \times C_{\partial_+ X} \ar[r]  & H_{n}   (\del X; \Z[0]) 
  },
\end{equation*}
where $-C_{\partial_- X}$ is the negative fundamental-cycle groupoid
of $\del_- X$, the full subcategory in $H_{n} (\del X; \Z[0])$ made up
by representatives of the negative fundamental class of $\del_- X$ in
$H_{n} (\del_- X; \Z) \subset H_{n} (\del X; \Z)$. By stacking
together the last two diagrams, we obtain the following diagram:
\begin{equation}
\label{little_triangle}
\xymatrix{
  C_{X, \partial X} \ar[r]^(.60){F} \ar[d]^{\del} \ruppertwocell<10>^0{\omit} 
  & \L \ltwocell<\omit>{<2.5>*!/^-2pt/{\labelstyle \Psi_F}}\\
  -C_{\partial_- X} \times C_{\partial_+ X} \ar[ur]_(.58)*!/^-1pt/{\labelstyle F_- + F_+}     \ar@{=>}(5,-7)*{};(9,-3)*{}_(0.7)*!/^3pt/{\labelstyle X_F}
},
\end{equation}
where $F_- := f|_{\del_- X}^* \alpha \cap -$ and $F_+ := f|_{\del_+
  X}^* \alpha \cap -$ appended by $H_0 (i)$ and $\deg$ as in
\eqref{big_triangle}.

Applying the limit functor, we get canonical morphisms
\[
-\lim_{C_{\del_- X}} F_- + \lim_{C_{\del_+ X}} F_+ \rightarrow
\lim_{-C_{\del_- X} \times C_{\del_+ X}} (F_-+F_+) \rightarrow
\lim_{C_{X,\del X}} F \to 0
\]
in $\L$, whence a morphism
\[
l_f: \lim_{C_{\del_- X}} F_- \rightarrow \lim_{C_{\del_+ X}} F_+,
\]
which translates into a canonical linear isometry between hermitian
lines.

\subsection{The Dijkgraaf-Witten theory TQFT functor}

Given a finite group $G$, for each $\alpha \in H^n(BG; \L)$,
we construct the Dijkgraaf-Witten theory TQFT functor,
\[
 Z^\alpha: \mathbf{Cob}(n+1) \to \Vect,
\]
from the category $\mathbf{Cob}(n+1)$ of cobordisms to the category
$\Vect$ of complex vector spaces, using the ingredients developed in
the preceding sections. We first construct the values of the functor
on objects. Observe that for every $Y \in \Ob \ \mathbf{Cob}(n+1)$,
Proposition~\ref{the_line} delivers a canonical hermitian line for
each $f \in \Map (Y, BG)$. We claim that these lines glue into a flat
hermitian line bundle over $\Map(Y, BG)$, or a \emph{local system}
with values in $\L$, \emph{i.e}., a functor
\[
\L_Y: \Pi_1\Map(Y, BG) \to \L
\]
from the fundamental groupoid of the mapping space $\Map(Y, BG)$ to
$\L$.

A morphism in $\Pi_1 \Map(Y, BG)$ is a homotopy class $[f]$ of a map
$f: Y \times I \to BG$. We can think of $Y \times I$ as the identity
cobordism between two copies of $Y$. Applying the construction of
Section \ref{morphism}, we get a morphism in $\L$,
\[
l_f: \lim_{C_Y} F_0 \rightarrow \lim_{C_Y} F_1.
\]
We define $\L_Y([f]) := l_f$.  The cocycle $f^* \alpha$ does depend on
the representative $f$ of the homotopy class $[f]$, see
Section~\ref{functoriality}, however the difference disappears at the
homology level after applying the cap product with $f^* \alpha$ and
the ``boundary homomorphism'' $\del_1: H_1 (Y \times I, \del (Y \times
I); \L) \to H_0 (\del (Y \times I); \L)$ in \eqref{big_triangle}. Note
that $f|_{\del (Y \times I)}^* \alpha$ does not depend on the
representative of the homotopy class $[f]$, because the homotopy is
supposed to be relative to the boundary. Thus, the diagram
\eqref{little_triangle} does not depend of the choice of a
representative of the homotopy class $[f]$, and the local system
$\L_Y$ is well defined.

One can view the construction of a local system $\L_y$ as
''integration of $\ev^* \alpha$ along fibers'' of $\pi$ or a
construction of the push-pull in cohomology with values in Picard
groupoids along the following diagram:
\[
\begin{CD}
Y \times \Map (Y,BG) @>\ev>> BG\\
@V{\pi}VV\\
\Map (Y, BG),
\end{CD}
\]
\[
H^n(BG; \L) \xrightarrow{\ev^*} H^n (Y \times \Map (Y,BG); \L)
\xrightarrow{\pi_*} H^0 (\Map (Y, BG); \L),
\]
where $\pi_* \ev^* \alpha := \L_Y$, by definition, and we recall that
objects of $H^0(\Map (Y, BG); \newline \L)$ are identified with local
systems or $0$-gerbes, see Section~\ref{0-gerbe}.

For any $Y \in \Ob \ \mathbf{Cob}(n+1)$, we define the value $Z^\alpha
(Y)$ of the \emph{TQFT functor} to be the space of global
sections of the local system $\L_Y$ over $\Map (Y, BG)$ constructed
above:
\[
Z^\alpha(Y) := H^0 (\Map(Y, BG); \L_Y) := \lim \L_Y \in \Vect,
\]
where the limit is taken for a natural extension $\Pi_1 \Map(Y, BG)
\xrightarrow{\L_Y} \L \to \Vect$ of the functor $\L_Y$, denoted by the
same symbol. The limit exists, because the category $\Vect$ is
complete.

Now we construct the arrow function of the TQFT functor. This can also
be viewed as a construction of ``fiberwise integral.'' Let $X$ be an
$(n+1)$-dimensional cobordism from $\del_- X$ to $\del_+ X$. We get
two local systems $\L_{\del_- X}$ and $\L_{\del_+ X}$ over the mapping
spaces $\Map(\del_- X, BG)$ and $\Map(\del_+ X, BG)$,
respectively. Let $p_{\pm}: \Map(X, BG) \to \Map(\del_\pm X, BG)$
denote the natural restriction morphisms. We start with constructing a
morphism $\L_X: p^*_{-} \L_{\del_- X} \to p^*_{+} \L_{\del_+ X} $ of
local systems on $\Map(X, BG)$. \emph{i.e}., a natural transformation
between functors $p^*_{-} \L_{\del_- X}$ and $p^*_{+} \L_{\del_+ X}:
\Pi_1(\Map(X, BG)) \to \L$. For each $f \in \Map(X, BG)$, by invoking
the construction of Section~\ref{morphism} once again, we get two
functors $F_\pm: C_{\del_\pm X} \to \L$ and the following morphism
\[
l_f: \lim_{C_{\del_- X}} F_- \to \lim_{C_{\del_+ X}} F_+
\]
in $\L$. Note that the fiber of each pull-back local system $p^*_{\pm}
\L_{\del_\pm X}$ over $f \in \Map(X, \linebreak[0] BG)$ is by
definition the fiber of $\L_{\del_\pm X}$ over $p_\pm(f)$, and that
fiber is $\lim_{C_{\del_\pm X}} F_\pm$ by the construction of
Section~\ref{object}. We define $\L_X(f)$ to be $l_f: p^*_{-}
\L_{\del_- X}|_f \to p^*_{+} \L_{\del_+ X}|_f $ on objects $f \in
\Map(X,BG)$ of $\Pi_1 (\Map (X,BG))$. A morphism $f \to g$ in the
fundamental groupoid $\Pi_1 (\Map (X,BG))$ is represented by a
homotopy $h \in \Map(X \times I, BG)$ between maps $f$ and $g \in \Map
(X, BG)$. To see that $\L_X$ consitutes a natural transformation, we
need to see that the diagram
\begin{equation}
\label{homotopy}
\begin{CD}
p^*_{-} \L_{\del_- X}|_f @>{l_f}>> p^*_{+} \L_{\del_+ X}|_f\\
@V{p^*_- l_{h|_{\del_- X \times I}}}VV @VV{p^*_+ l_{h|_{\del_- X \times I}}}V\\
p^*_{-} \L_{\del_- X}|_g @>{l_g}>> p^*_{+} \L_{\del_+ X}|_g
\end{CD}
\end{equation}
commutes. Indeed, the homotopy gives a morphism $H: f^* \alpha \to g^*
\alpha$ in the Picard groupoid $H^n (X; \L)$. Using the
bifunctoriality of the cap product, we get a 2-morphism $f^* \alpha
\cap - \Rightarrow g^* \alpha \cap - $ added to Diagram \eqref{F-map},
resulting in a commutative triangle
\begin{equation*}
\xymatrix{
F \ar@{=>}[rr]^{\Psi_H} \ar@{=>}[dr]_{\Psi_F}& & G \ar@{=>}[dl]^{\Psi_G}\\
& 0
}
\end{equation*}
on top of the upper part of Diagram~\eqref{little_triangle} and,
similarly, a commutative square
\begin{equation*}
  \xymatrix{
    (F_- + F_+) \circ \del \ar@{=>}[r]^{\Psi_{\del H}} \ar@{=>}[d]_{X_F}
    & (G_- + G_+) \circ \del \ar@{=>}[d]^{X_G}\\
    F \ar@{=>}[r]^{\Psi_H}  & G
  }
\end{equation*}
on top of the lower part of Diagram~\eqref{little_triangle}, with
$\Psi_{\del H}$ coming from the 2-morphism $f|_{\del X}^* \alpha \cap
- \Rightarrow g|_{\del X}^* \alpha \cap - $ added to the bottom
triangle in \eqref{F-map}. Passing to the limits, we see that
\eqref{homotopy} is commutative.


Now, after the morphism $\L_X: p^*_{-} \L_{\del_- X} \to p^*_{+}
\L_{\del_+ X} $ of local systems on $\Map(X, \linebreak[0] BG)$ is
constructed, we are ready to construct a linear map
\[
Z^\alpha (X): Z^\alpha (\del_- X) \to Z^\alpha (\del_+ X)
\]
or
\[
Z^\alpha (X): H^0 (\Map (\del_- X, BG); \L_{\del_- X}) \to H^0 (\Map
(\del_+ X, BG); \L_{\del_+ X}).
\]
The plan is to describe a push-pull along the diagram of spaces:
\[
\Map (\del_- X, BG) \xleftarrow{p_-} \Map (X, BG) \xrightarrow{p_+}
\Map (\del_+ X, BG).
\]
The pullback
\[
p_-^*: H^0 (\Map (\del_- X, BG); \L_{\del_- X}) \to H^0 (\Map (X, BG);
p_-^* \L_{\del_- X})
\]
is easy. So is an intermediate map: 
\[
H^0( \L_X): H^0 (\Map (X, BG); p_-^* \L_{\del_- X}) \to H^0 (\Map (X,
BG); p_+^* \L_{\del_+ X}).
\]
The pushforward
\[
(p_+)_*: H^0 (\Map (X, BG); p_+^* \L_{\del_+ X}) \to H^0 (\Map (\del_+
X, BG); \L_{\del_+ X})
\]
is not straightforward, and its existence relies on the specifics of
the topology of mapping spaces to $BG$ for a finite group $G$.

Recall that the space $\Map(X,BG)$ may naturally be realized at the
classifying space for principal $G$-bundles over $X$. This leads to a
natural homotopy equivalence
\[
\Map (X, BG) \sim \coprod_{[P \to X]} B \Aut (P),
\]
where the disjoint union is taken over isomorphism classes $[P \to X]
\simeq \pi_0 \Map (X, \linebreak[0] BG)$ of principal $G$-bundles $P
\to X$. The map $p_+: \Map (X, BG) \to \Map (\del_+ X, BG)$ is
homotopy equivalent to the natural restriction map
\[
p'_+: \coprod_{[P \to X]} B \Aut (P) \to \coprod_{[P_+ \to \del_+ X]} B
\Aut (P_+),
\]
which is a finite covering map over each connected component $B \Aut
(P_+)$, sometimes with empty fiber.

We will define the pushforward
\[
(p_+)_*: H^0 (\Map (X, BG); p_+^* \L_{\del_+ X}) \to H^0 (\Map (\del_+
X, BG); \L_{\del_+ X})
\]
as a transfer map
{\small
\[
(p'_+)_*: H^0 \left( \coprod_{[P \to X]} B \Aut (P); (p'_+)^*
  \L_{\del_+ X}\right) \to H^0 \left( \coprod_{[P_+ \to \del_+ X]} B
  \Aut (P_+); \L_{\del_+ X} \right),
\]
}
which will be constructed using the definition of $H^0$ as a limit
over the fundamental groupoid. Indeed, for every path $\gamma_+$ in $B
\Aut (P_+)$, we take all its lifts to the component $B \Aut (P)$ over
$B \Aut (P_+)$, which is a finite, possibly zero, number. For each
such path $\gamma$, we have a linear isometry $(p'_+)^* \L_{\del_+ X}
(\gamma): (p'_+)^* \L_{\del_+ X}(\gamma(0)) \to (p'_+)^* \L_{\del_+
  X}(\gamma(1))$, which, by definition of $(p'_+)^*$, is equal to the
isometry $\L_{\del_+ X} (\gamma_+): \L_{\del_+ X}(\gamma_+(0)) \to
\L_{\del_+ X}(\gamma_+(1))$. Since $H^0 \left( \coprod_{[P \to X]} B
  \Aut (P); (p'_+)^* \L_{\del_+ X} \right)$ is a limit of the functor
$(p'_+)^* \L_{\del_+ X}$, we have a canonical commutative diagram of
linear maps:
{\small
\[
\xymatrix{
    & H^0 \left( \coprod_{[P \to X]} B \Aut (P); (p'_+)^* \L_{\del_+ X} \right) \ar[dl] \ar[dr]\\
    (p'_+)^* \L_{\del_+ X}(\gamma(0)) \ar[rr] \ar@{=}[d]& & (p'_+)^*
    \L_{\del_+
      X}(\gamma(1)) \ar@{=}[d]\\
    \L_{\del_+ X}(\gamma_+(0)) \ar[rr] & & \L_{\del_+ X}(\gamma_+(1)).
  } 
  \]}
\begin{sloppypar}
If, given $\gamma_+$, we add the linear maps $H^0 \left( \coprod_{[P
    \to X]} B \Aut (P); (p'_+)^* \L_{\del_+ X} \right) \linebreak[3]
\to \linebreak[4] \L_{\del_+ X}(\gamma_+(0))$ over all possible
$\gamma$'s covering $\gamma_+$ and do the same for maps to $\L_{\del_+
  X}(\gamma_+(1))$, we will get a commutative diagram
\end{sloppypar}
{\small
\[
\xymatrix{
    & H^0 \left( \coprod_{[P \to X]} B \Aut (P); (p'_+)^* \L_{\del_+ X} \right) \ar[dl] \ar[dr]\\
    \L_{\del_+ X}(\gamma_+(0)) \ar[rr] & & \L_{\del_+ X}(\gamma_+(1)).
 }
\]
}
Since $H^0 \left( \coprod_{[P_+ \to \del_+ X]} B \Aut (P_+);
  \L_{\del_+ X} \right)$ is a limit of the functor $\L_{\del_+ X}$, we
get a canonical linear map
{\small
\[
H^0 \left( \coprod_{[P \to X]} B \Aut (P); (p'_+)^* \L_{\del_+ X}
\right) \to H^0 \left( \coprod_{[P_+ \to \del_+ X]} B \Aut (P_+);
  \L_{\del_+ X} \right),
\]
}
which we declare to be the \emph{transfer} $(p'_+)_*$.

Finally, the TQFT functor
\[
Z^\alpha (X): H^0 (\Map (\del_- X, BG); \L_{\del_- X}) \to H^0 (\Map
(\del_+ X, BG); \L_{\del_+ X})
\]
is defined as the composition of $(p_+)_*$, $H^0(\L_X)$, and $p_-^*$.

The invariance of $Z^\alpha$ under diffeomorphisms $X' \to X''$ of
cobordisms is obvious, as a diffeomorphism induces an isomorphism of
simplicial sets $\Sing(X')$ and $\Sing(X'')$ representing the
cobordisms and leads to isomorphic diagrams \eqref{big_triangle} and
\eqref{little_triangle} in a strict sense, thus giving the same
isometry $l_f$ of hermitian lines in Section~\ref{morphism}.

\appendix
\section{The Hom $2$-chain complex}
In this section we define the Hom $2$ - chain complex and a
tensor product in $\mathbf{2Ch({\emph{SCG}}})$. We recall that
given any two Picard groupoids $\A, \B \in \Ob(\emph{SCG})$,
$\Hom_{\emph{SCG}}(\A, \B)$ inherits a Picard groupoid structure,
i.e. the category $\emph(SCG)$ is enriched over itself .
Let $A_{\bullet}, \B_{\bullet} \in
\Ob(\mathbf{2Ch({\emph{SCG}}}))$. Then,
$(\Hom_{\mathbf{2Ch({\emph{SCG}}})}(\A_{\bullet}, \B_{\bullet}), d,
\phi)$ is a
chain complex whose $nth$. degree is defined as follows:\\
$\Hom_{\mathbf{2Ch({\emph{SCG}}})}(\A_{\bullet}, \B_{\bullet})_{n} = \underset{p} \prod \Hom_{\emph{SCG}}(\A_{p},\B_{p+n})$.\\
The differential $d:\Hom_{\mathbf{2Ch({\emph{SCG}}})}(\A_{\bullet},
\B_{\bullet})_{n} \rightarrow
\Hom_{\mathbf{2Ch({\emph{SCG}}})}(\A_{\bullet}, \B_{\bullet})_{n-1}$
is given by $(df)_{p} = df_{p} +(-1)^{p+1}f_{p-1}d$ and a composition
of $2$ - morphisms $dd(f) \Rightarrow d^{2}f + fd^{2} \Rightarrow 0$,
where the first $2$ - morphism comes from the distributivity law on
each degree of the Hom complex which is the consequence of the
enrichment of $\emph{SCG}$ over itself and the second $2$ - morphism
in the composition is obvious.  Similarly, we may define the Hom 2 -
chain complex of chain maps between a 2-cochain complex and a 2-chain
complex. Let $\C^{\bullet}$ be a 2-cochain complex of Picard
groupoids, let $\B_{\bullet} \in \Ob(\mathbf{2Ch({\emph{SCG}})})$,
then the 2-chain complex $([\C^{\bullet}, \B_{\bullet}], d, \phi)$ is
defined as the 2 - chain complex
$(\Hom_{\mathbf{2Ch({\emph{SCG}}})}(\C^{-\bullet}, \B_{\bullet}), d,
\phi)$, where $\C^{-\bullet} \in \Ob(\mathbf{2Ch({\emph{SCG}}}))$ is
the 2-chain complex obtained by negatively regrading $\C^{\bullet}$,
its degree $n$ is $[\C^{\bullet}, \B_{\bullet}]_{n} = \underset{p}
\prod \Hom_{\emph{SCG}}(\C^{-p},\B_{p+n})$.  The tensor product of two
chain complexes could be defined similarly.

\section{$\Pic$ categories}
\label{pic-category-appendix}

In this appendix we give the definition of a mathematical structure
which is built on a bicategory but whose mapping categories have
the structure of Picard groupoids.
\begin{df}
 A $\Pic$-category $\C$ consists of the following data
 \begin{enumerate}
  \item A small set, $\Ob(\C)$, whose elements will be called the
  \emph{objects} of $\C$.
  \item A function $\C(-,-):\Ob(\C) \times \Ob(\C) \to \Ob(\Pic)$,
  where $\Ob(\Pic)$ is the set of all Picard groupoids.
  \item For each object $s \in \Ob(\C)$, a homomorphism $id_s:\ast
  \to \C(s,s)$, where $\ast$ is the terminal Picard groupoid.
  \item For each triple of objects $s, t, u \in \Ob(\C)$, a \emph{composition
  bifunctor} $-\circ -:\C(t,u) \times \C(s,t) \to \C(s,u)$ which is subject
  to the following conditions
  \begin{enumerate}
  \item For each $h \in \Ob(\C(t,u))$, the functor
  \[
   h \circ -:\C(s,t) \to \C(s,u).
  \]
   is a homomorphism
  \item For each $g \in \Ob(\C(s,t))$, the functor
  \[
   - \circ g:\C(t,u) \to \C(s,u).
  \]
  is a homomorphism.
  \end{enumerate}  
  \item For each triple of objects $s, t, u \in \Ob(\C)$ and each pair of
  morphisms $g_1, g_2 \in \Ob(\C(s,t))$, a monoidal natural transformation
  $\phi^-_{g_1,g_2}:-\circ g_1 + - \circ g_2 \Rightarrow - \circ g_1+g_2$,
  where the homomorphism $-\circ g_1 + - \circ g_2:\C(t,u) \to \C(s,u)$ is
  defined pointwise.
  \item For each triple of objects $s, t, u \in \Ob(\C)$ and each pair of
  morphisms $h_1, h_2 \in \Ob(\C(t,u))$, a monoidal natural transformation
  $\psi_-^{h_1,h_2}: h_1 \circ - + h_2 \circ - \Rightarrow h_1+h_2 \circ -$,
  where the homomorphism $h_1 \circ - + h_2 \circ -:\C(s,t) \to \C(s,u)$ is
  defined pointwise.
  \item For each quadruple of objects $s,t,u,v \in \Ob(\C)$, a natural
  transformation, $\alpha$ called the \emph{associator},
  between functors defined in the following diagram
  \begin{equation*}
  \xymatrix{
    \C(u,v) \times \C(t,u) \times \C(s,t)
    \ar[r]^{\ \ \ \ \ \ id \times -\circ-} \ar[d]_{-\circ- \times id} &
    \C(u,v) \times \C(s,u) \ar[d]^{-\circ-} \ar@{}[dl]^{\alpha}|{\Leftarrow}\\
    \C(t,v) \times \C(s,t) \ar[r]  & \C(s,v). 
  },
 \end{equation*}
 and which is subject to the following conditions:
 \begin{enumerate}
  \item For each pair $(g, h) \in \Ob(\C(t,u)) \times \Ob(\C(u,v))$,
  the natural transformation $\alpha_{(h,g,-)}$ as in
  the following diagram
  \begin{equation*}
  \xymatrix{ 
  \C(s,t) \ar[r]_{g \circ -} \rruppertwocell<12>^{(h \circ g) \circ -}{\omit}
    & \C(s,u) \ar[r]_{h \circ -}  & \C(s,v) \lltwocell<\omit>{<3>\alpha_{(h,g,-)} \ \ \ \ \ \ \ \ }
  }
 \end{equation*}
  is a monoidal natural transformation.
  \item For each pair $(f, h) \in \Ob(\C(s,t)) \times \Ob(\C(u,v))$,
  the natural transformation $\alpha_{(h,-,f)}$ as in
  the following diagram
  \begin{equation*}
  \xymatrix{
    \C(t,u) \ar[r]^{h \circ-} \ar[d]_{-\circ f} &
    \C(t,v) \ar[d]^{- \circ f} \ar@{}[dl]^{\alpha_{(f,-,h)}}|{\Leftarrow}\\
    \C(s,u) \ar[r]_{h \circ-}  & \C(s,v) 
  }
 \end{equation*}
  is a monoidal natural transformation.
  \item For each pair $(f, g) \in \Ob(\C(s,t)) \times \Ob(\C(t,u))$,
  the natural transformation $\alpha_{(-,g,f)}$ as in
  the following diagram
  \begin{equation*}
  \xymatrix{ 
  \C(u,v) \ar[r]^{- \circ g} \rrlowertwocell<-12>_{- \circ (g \circ f)}{\omit}
    & \C(t,v) \ar[r]^{- \circ f}  & \C(s,v) \lltwocell<\omit>{<-3>\alpha_{(-,g,f)} \ \ \ \ \ \ \ \ }
  }
 \end{equation*}
  is a monoidal natural transformation.
  
 \end{enumerate}
 \item For each pair of objects $s,t \in \Ob(\C)$, two monoidal
 natural transformations
 \begin{equation*}
  \xymatrix{ 
  \C(s,t) \ar@{=}[r] \rlowertwocell<-10>_{id_t \circ -}{\omit}
    & \C(s,t) \ltwocell<\omit>{<-2>\lambda } &&
  \C(s,t) \ar@{=}[r] \rlowertwocell<-10>_{- \circ id_s}{\omit}
    & \C(s,t) \ltwocell<\omit>{<-2>\rho }
  }
 \end{equation*}

 \end{enumerate}

 \end{df}
 Let $\C$ and $\D$ be two $\Pic$-categories, A functor of
 $\Pic$-categories $F:\C \to \D$ is a functor of bicategories
 which respects the additional structure on the morphism categories
 of $\C$ and $\D$. We will skip a precise definition of a functor
 of $\Pic$-categories but an interested reader can define these functors
 rigorously using our definition of $\Pic$-categories.

\bibliographystyle{amsalpha}
\bibliography{dwf}

\end{document}